\documentclass[11pt,reqno,oneside]{amsart}

\usepackage[letterpaper,hmargin=1.35in,vmargin=1.22in]{geometry}
\usepackage{natbib}
\usepackage{setspace}
\usepackage{amssymb}
\usepackage{tikz}
\usetikzlibrary{matrix,arrows,cd,calc}
\usepackage{mathtools}
\usepackage{xspace}
\usepackage[hyphens]{url}
\usepackage{hyperref}
\hypersetup{pdftex,colorlinks=true,allcolors=blue}
\usepackage{hypcap}
\usepackage{color}
\definecolor{darkgreen}{rgb}{0,0.45,0}
\hypersetup{colorlinks,urlcolor=blue,citecolor=darkgreen,linkcolor=darkgreen,linktocpage}
\usepackage[capitalize]{cleveref}

\numberwithin{equation}{section}

\newtheorem{defn}{Definition}[section]
\newtheorem*{defn*}{Definition}
\newtheorem{lem}[defn]{Lemma}
\newtheorem{prop}[defn]{Proposition}
\newtheorem{cor}[defn]{Corollary}
\newtheorem{thm}[defn]{Theorem}

\theoremstyle{remark}
\newtheorem{rmk}[defn]{Remark}
\newtheorem*{rmk*}{Remark}
\newtheorem{eg}[defn]{Example}

\newcommand{\define}[1]{\textbf{\boldmath{#1}}}
\newcommand{\defeq}{\vcentcolon\equiv}

\newcommand{\sbt}{\bullet}
\DeclarePairedDelimiter{\ceil}{\lceil}{\rceil}
\newcommand*\sq{\mathbin{\vcenter{\hbox{\rule{.3ex}{.3ex}}}}}
\newcommand{\isPrincipal}{\mathsf{isPrincipal}}
\newcommand{\isKGnbun}{\mathsf{is\text{-}}K(G,n)\mathsf{\text{-}bundle}}
\newcommand{\north}{\mathsf{N}}
\newcommand{\south}{\mathsf{S}}

\newcommand{\coKer}{\mathsf{coker}}
\newcommand{\Aut}{\mathsf{Aut}}
\newcommand{\Grp}{\mathsf{Grp}}
\newcommand{\act}{\mathsf{act}}
\newcommand{\EM}{\mathsf{EM}}
\newcommand{\loopspacesym}{\mathsf{\Omega}}
\newcommand{\loopspace}[1]{\loopspacesym(#1)}
\newcommand{\suspsym}{\mathsf{\Sigma}}
\newcommand{\susp}[1]{\suspsym #1}
\newcommand{\fib}[2]{{\mathsf{fib}}_{#1}(#2)}
\newcommand{\fibb}[1]{{\mathsf{fib}}_{#1}}

\newcommand{\degg}{\mathsf{deg}}

\newcommand{\UU}{\mathcal{U}}

\newcommand{\refl}[1]{\ensuremath{\mathsf{refl}_{#1}}\xspace}
\newcommand{\merid}{\ensuremath{\mathsf{merid}}\xspace}
\newcommand{\transfib}[3]{\ensuremath{\mathsf{transport}^{#1}(#2,#3)\xspace}}

\newcommand{\unit}{\ensuremath{\mathbf{1}}\xspace}

\newcommand{\decode}{\ensuremath{\mathsf{decode}}\xspace}

\newcommand{\N}{\mathbb{N}}

\newcommand{\idfunc}[1]{\mathsf{id}_{#1}}
\newcommand{\ttrunc}[2]{\bigl\Vert #2\bigr\Vert_{#1}}
\newcommand{\ap}[1]{\mathsf{ap}_{#1}}

\newcommand{\precomp}[1]{{#1}^*}

\newcommand{\true}{\mathsf{true}}

\newcommand{\longhookrightarrow}{\lhook\joinrel\lra}
\newcommand{\s}{\mathbf{s}}
\newcommand{\rr}{\mathbf{r}}
\newcommand{\lra}       {\longrightarrow}
\newcommand{\llra}[1]   {\stackrel{#1}{\lra}}

\makeatletter
\def\smsym{\sum}
\newcommand{\@thesum}[1]{\smsym_{(#1)}}
\newcommand{\sm}[1]{\@ifnextchar\bgroup{\@sm{#1}\sm}{\@sm{#1}}}
\newcommand{\@sm}[1]{\mathchoice{{\textstyle\@thesum{#1}}}{\@thesum{#1}}{\@thesum{#1}}{\@thesum{#1}}}

\def\prdsym{\prod}
\newcommand{\@ifnextchar@starorbrace}[2]
  {\@ifnextchar*{#1}{\@ifnextchar\bgroup{#1}{#2}}}

\newcommand{\@theprd}[1]{\prdsym_{(#1)}}
\newcommand{\@theiprd}[1]{\prdsym_{\{#1\}}}

\newcommand{\prd}{\@ifnextchar*{\@iprd}{\@prd}}
\newcommand{\@prd}[1]
  {\@ifnextchar@starorbrace
    {\@tprd{#1}\prd}
    {\@tprd{#1}}}
\newcommand{\@tprd}[1]{%
  \mathchoice{%
    {{\textstyle\@theprd{#1}}}}{\@theprd{#1}}{\@theprd{#1}}{\@theprd{#1}}}

\newcommand{\@iprd}[2]{\@ifnextchar@starorbrace%
  {\@tiprd{#2}\prd}%
  {\@tiprd{#2}}}
\newcommand{\@tiprd}[1]{
  \ifthenelse{\true}
    {\@@tiprd{#1}}
    {\@tprd{#1}}}
\newcommand{\@@tiprd}[1]{\mathchoice{{\textstyle\@theiprd{#1}}}{\@theiprd{#1}}{\@theiprd{#1}}{\@theiprd{#1}}}
\newcommand{\eqvsym}{\simeq}    
\newcommand{\eqv}[2]{
  \@ifnextchar\bgroup
    {#1 \eqvsym \eqv{#2}}
    {#1 \eqvsym #2}
  }
\makeatother

\begin{document}

\title{Nilpotent Types and Fracture Squares in Homotopy Type Theory}

\author[L. Scoccola]{Luis Scoccola}


\begin{abstract}
    We develop the basic theory of nilpotent types and their localizations away from sets of numbers in Homotopy Type Theory.
    For this, general results about the classifying spaces of fibrations with fiber an Eilenberg--Mac Lane space are proven.
    We also construct fracture squares for localizations away from sets of numbers.
    All of our proofs are constructive.
\end{abstract}

\maketitle

\tableofcontents

\section{Introduction}

Nilpotent spaces play a very important role in the homotopy theory of spaces.
Many results that hold for simply connected spaces can be generalized to nilpotent spaces.
For example, any cohomology isomorphism between nilpotent spaces is a weak equivalence.
Nilpotent spaces also have a rich theory of localizations away from sets of numbers, including
fracture squares that reconstruct a space out of some of its localizations.

Nilpotent spaces are characterized by the fact that the maps in their Postnikov tower factor as finite composites of principal fibrations.
Nevertheless, they are usually defined in terms of some algebraic structure on the actions of their fundamental
group on their homotopy groups.
Both characterizations are useful, and one of the basic theorems in the theory is the fact that the
characterizations are equivalent.

In this paper we develop the basic theory of nilpotent spaces in Homotopy Type Theory.
In \cref{section:nilptypes} we prove the equivalence between the two characterizing properties of nilpotency (\cref{theorem:nilpotentchar}).
In order to do this, we study the relationship between unpointed Eilenberg--Mac Lane spaces and doubly-pointed
Eilenberg--Mac Lane spaces, and prove that the type of unpointed $n$-dimensional Eilenberg--Mac Lane spaces
is equivalent to the type of doubly-pointed $(n+1)$-dimensional Eilenberg--Mac Lane spaces (\cref{proposition:characterizationF}).
In \cref{section:cohiso} we follow the suggestion in \citep{shulman}, and we prove that cohomology isomorphisms
between nilpotent types induce isomorphisms in all homotopy groups.
In \cref{section:localization} we study the localization of nilpotent types and its effect on homotopy groups.
The main result of this section is that localization of a nilpotent type localizes its homotopy groups in the expected way (\cref{theorem:localizationlocalizes}).
Finally, in \cref{section:fracturesquares} we construct fracture squares for localizations away from sets of numbers.
In particular, we construct a fracture square for simply connected types that
avoids assuming Whitehead's principle (\cref{theorem:fracturetheorem}).
This construction is very different from the classical one that can be found in, e.g., \citep{MayPonto}.
For nilpotent types we only construct a square in the case the type is truncated.
The last two sections build on the results of \citep{RSS} and \citep{CORS}.

The proofs of the main results in this paper are different from the classical proofs (see the treatment of
nilpotent spaces and their localizations in \citep{MayPonto}, \citep{bousfieldkan}, \citep{locnilp}).
As we work in the type theory described in \citep{hottbook},
avoiding the use of Whitehead's theorem \citep[Section~8.8]{hottbook}, we get constructive proofs that work in any $\infty$-topos
\citep{KL}, \citep{LS}, \citep{shulmantopoi}, assuming the initiality conjecture for Homotopy Type Theory.
What makes the theory of nilpotent types work in this setting is
\cref{proposition:characterizationF}, which gives a hands-on characterization of $K(A,n)$-bundles, and
a simple definition of the action $\pi_1(X) \curvearrowright A$ associated to a pointed $K(A,n)$-bundle
$Y \to X$. This is inspired by \citep{shulman}, which discusses the advantages of working directly
with the classifying spaces of $K(A,n)$-bundles.

\medskip

\noindent
\textbf{Acknowledgements:}
I would like to thank Dan Christensen and Egbert Rijke for helpful conversations about this project,
and for comments during the writing of this paper.
I would also like to thank the anonymous referees for detailed comments and suggestions.

\medskip
\subsection{Conventions and Background}
We work in Homotopy Type Theory as described in \citep{hottbook}.
That is, intensional Martin-L\"of dependent type theory together with the univalence axiom and
higher inductive types.

We use the existence and universal property of Eilenberg--Mac Lane spaces \citep{FinsterLicata}, and
the existence of localizations at families of maps \citep{RSS}, as well as many of their properties in the
special case of localizations at self-maps of the circle \citep{CORS}.

\section{Nilpotent types}
\label{section:nilptypes}

In this section we develop the basic theory of nilpotent types.
The main result is \cref{theorem:nilpotentchar}, which characterizes nilpotent spaces
in terms of certain factorizations of the maps in their Postnikov towers.
Although the main result is known classically, most of the proofs we give are quite different.
In particular, we give a description of
fibrations with fiber an Eilenberg--Mac Lane space (\cref{remark:descriptionbundles}),
which allows us to prove a generalization of the classification of nilpotent types (\cref{theorem:mainchar}).

In \cref{grouptheory} we show that the theory of groups and group actions can be seen as a particular
case of the theory of spaces and fibrations. In \cref{Kbundles} we study $K(A,n)$-bundles, by studying their
classifying space. We prove a general stabilization result about these classifying spaces (\cref{proposition:characterizationF})
which allows us to give a workable description of $K(A,n)$-bundles.
In \cref{principalfibrations} we discuss principal $K(A,n)$-bundles, and give a characterization in terms of group actions.
Finally, in \cref{nilpgroups} and \cref{nilpspaces}, we introduce nilpotent groups and types and characterize them
in terms of group actions.

\subsection{Homotopical interpretation of group theory}
\label{grouptheory}

    We start by setting up some basic notation and results about group theory and its interpretation
    as a particular case of the theory of Eilenberg--Mac Lane spaces.
    We will rely on several results in \citep{BRV}.

    Throughout the section, $G$ and $H$ will denote arbitrary groups.
    As is usual in Homotopy Type Theory, we assume that the underlying type of a group is a set.
    The type of maps between groups $G$ and $H$ will be denoted by $G \to_\Grp H$.

    The type $K(G,n)$ will denote the Eilenberg--Mac Lane space of type $G,n$ as defined in \citep{FinsterLicata}.
    If $n>1$, when we write $K(G,n)$ we will be assuming that $G$ is abelian.
    When regarded as a pointed type, the pointing of $K(G,n)$ will be denoted by $\ast : K(G,n)$,
    and the inclusion of the point by $\iota : \unit \to K(G,n)$.
    When regarded as a doubly-pointed type, we will use the point $\ast : K(G,n)$ twice.
    The first building block in the homotopical interpretation of group theory is the following result.

    \begin{thm}[see {\citep[Theorem~5.1]{BRV}}]
        \label{thm4}
    The map $\loopspacesym^n : \UU_\sbt \to \UU_\sbt$ restricts to an equivalence
    of categories between pointed, connected, $1$-truncated types and groups, when $n = 1$;
    and between pointed, $(n-1)$-connected, $n$-truncated types and abelian groups,
    when $n \geq 2$.
    \end{thm}

    \begin{defn}
    Let $G$ and $H$ be groups and let $\Aut(H)$ be the group of group automorphisms of $H$.
    An \define{action} of $G$ on $H$ is a group morphism $G \to_{\Grp} \Aut(H)$.
        The type of actions of $G$ on $H$ is denoted by $G \curvearrowright H$.
    \end{defn}

    \begin{defn}
        The \define{trivial} action, denoted by $\unit : G \curvearrowright H$, is the trivial
        group homomorphism $G \to_\Grp \Aut(H)$ given by $g \mapsto \idfunc{H}$.
        We say that an action is trivial if it is equal to the trivial action.
    \end{defn}

    \begin{defn}\label{definition:mapofactionsdef}
        Given groups $G$, $H$, $H'$, and actions $\alpha : G \curvearrowright H$, $\alpha' : G \curvearrowright H'$,
        a \define{map of actions} $\phi : \alpha \to_{\act} \alpha'$ is given by a group morphism $\phi : H \to_\Grp H'$
        such that for every $g : G$ and every $h : H$, we have $\phi(\alpha(g,h)) = \alpha'(g,\phi(h))$.
        We say that such a map is a \define{subaction} if its underlying group morphism is a monomorphism.
        Similarly, a map of actions is \define{surjective} if its underlying group morphism is an epimorphism.
    \end{defn}
    
    Notice that for a group morphism $\phi : H \to_{\Grp} H'$, being a map of actions is a mere property, since respecting
    the action is specified by a family of equalities in a set.

    The following lemma follows from \cref{thm4}.

    \begin{lem}\label{lemma:replacebyKG1}
            Let $X$ be a pointed, connected type. Then, using the identity morphism $\pi_1(X) \to_\Grp \pi_1(X)$ in
            the induction principle of $K(\pi_1(X),1)$ \citep[Section~3.1]{FinsterLicata}
        we get a pointed map $K(\pi_1(X),1) \to_\sbt \ttrunc{1}{X}$. This map is an equivalence.
    \qed
    \end{lem}

\begin{defn}
    Given a natural number $n\geq 1$, and a group $G$ (assumed to be abelian if $n>1$),
    define the \define{type of pointed Eilenberg--Mac Lane spaces} of type $G$, $n$ as follows:
    \[
        \EM_\sbt(G,n) \defeq \sum_{(K:\UU_\sbt)}\ttrunc{-1}{K =_{\UU_\sbt} K(G,n)}.
    \]
\end{defn}

    The type $\EM_\sbt(G,n)$ classifies group actions on $G$,
    in the following sense.

    \begin{lem}\label{remark:equivalenceofEM}
        Given $n \geq 1$ and $K,K' : \EM_\sbt(G,n)$, the map
        $(K =_{\UU_\sbt} K') \to (\loopspacesym^n K =_\Grp \loopspacesym^n K')$ 
        given by the functoriality of $\loopspacesym$ is an equivalence.
        This equivalence provides us with an equivalence $\EM_\sbt(G,n) \simeq K(\Aut(G),1)$ since $\EM_\sbt(G,n)$
        is pointed and connected, and its loop space is $K(G,n) =_{\UU_\sbt} K(G,n)$.
    \end{lem}
    \begin{proof}
        The first statement follows from \cref{thm4}.
        The second statement follows from the first one.
    \end{proof}

    \begin{lem}\label{lemma:actionisaction}
        Given a pointed, connected type $X$, an integer $n\geq 1$, and a group $G$, we have a map
        \[
          \left(X \to_\sbt \EM_\sbt(G,n)\right) \to \left(\pi_1(X) \curvearrowright G\right)
        \]
        given by functoriality of $\pi_1$ and the fact that $\EM_\sbt(G,n) \simeq K(\Aut(G),1)$.
        This map is an equivalence.
        Moreover, an action $\alpha : \pi_1(X) \curvearrowright G$ is trivial precisely when its
        corresponding map $\alpha : X \to \EM_\sbt(G,n)$ is null-homotopic.
    \end{lem}
    \begin{proof}
    Using the fact that $\ttrunc{1}{X} \simeq K(\pi_1(X),1)$, together with \cref{remark:equivalenceofEM},
    we get the following chain of equivalences:
    \begin{align*}
        \left(X \to_\sbt \EM_\sbt(G,n)\right) &\simeq \left(\ttrunc{1}{X} \to_\sbt \EM_\sbt(G,n)\right)\\
                                              &\simeq \left(\pi_1(X) \to_{\Grp} \Aut(G)\right)\\
                                              &\equiv \left(\pi_1(X) \curvearrowright G\right),
    \end{align*}
        where the second equivalence follows from \cref{thm4}.

    The second claim follows from the fact that the constant map $X \to \EM_\sbt(G,n)$ induces the trivial action.
    \end{proof}

    Whenever $X$ is pointed and connected, and we are given an action
    $\alpha : \pi_1(X) \curvearrowright G$, we denote its associated map
    of type $X \to_\sbt \EM_\sbt(G,n)$ by $\hat{\alpha}$.
    When there is no risk of confusion, we will abuse notation and denote both with the same symbol.

    This characterization is functorial, in the following sense.

    \begin{lem}
        Given a pointed, connected type $X$, an integer $n\geq 1$, groups $H,H'$, and actions
        $\alpha : \pi_1(X) \curvearrowright H$ and $\alpha' : \pi_1(X) \curvearrowright H'$,
        there is an equivalence between the type of maps of actions $\alpha \to_\act \alpha'$,
        and the type of fiberwise pointed maps $\prd{x : X} \hat{\alpha}(x) \to_\sbt \hat{\alpha}'(x)$.
    \end{lem}
    \begin{proof}
        The maps $\hat{\alpha} : X \to_\sbt \EM_\sbt(H,n)$ and $\hat{\alpha}' : X \to_\sbt \EM_\sbt(H',n)$
        factor through the truncation $X \to \ttrunc{1}{X}$, since both $\EM_\sbt(H,n)$ and $\EM_\sbt(H',n)$
        are $1$-truncated. Let us refer to these extensions as $\beta$ and $\beta'$ respectively.
        Moreover, for every $x : X$, the type $\hat{\alpha}(x) \to_\sbt \hat{\alpha}'(x)$ is $1$-truncated, by \cref{thm4},
        so $\prd{x : X} \hat{\alpha}(x) \to_\sbt \hat{\alpha}'(x)$ is equivalent to $\prd{x : \ttrunc{1}{X}} \beta(x) \to_\sbt \beta'(x)$.
        Now, by \cref{lemma:replacebyKG1}, it is enough to provide an equivalence
        \[
            \left( \prd{x : X'} \hat{\alpha}(x) \to_\sbt \hat{\alpha}'(x) \right) \simeq (\alpha \to_\act \alpha')
        \]
        where we have implicitly used the equivalence $X \simeq X' \equiv K(\pi_1(X),1)$ of the lemma.

        This reduction lets us use the induction principle of $K(\pi_1(X),1)$ \citep[Section~3.1]{FinsterLicata}.
        In this case, the principle tells us that the type $\prd{x : X'} \hat{\alpha}(x) \to_\sbt \hat{\alpha}'(x)$
        is equivalent to the type of maps $\phi : \hat{\alpha}(\ast) \to_\sbt \hat{\alpha}'(\ast)$
        together with a proof that for every $l : \ast = \ast$, we have $\transfib{\hat{\alpha}(-) \to_\sbt \hat{\alpha}'(-)}{l}{\phi} = \phi$.
        Now, by path-induction, this last type is equivalent to the type of proofs that the following square is pointed commutative:
        \[
        \begin{tikzpicture}
          \matrix (m) [matrix of math nodes,row sep=2em,column sep=2em,minimum width=2em,nodes={text height=1.75ex,text depth=0.25ex}]
            { \hat{\alpha}(\ast) & \hat{\alpha}'(\ast) \\
            \hat{\alpha}(\ast) & \hat{\alpha}'(\ast). \\};
          \path[-stealth]
            (m-1-1) edge node [above] {$\phi$} (m-1-2)
                    edge node [left] {$\transfib{\hat{\alpha}}{l}{-}$} (m-2-1)
            (m-2-1) edge node [above] {$\phi$} (m-2-2)
            (m-1-2) edge node [right]{$\transfib{\hat{\alpha}'}{l}{-}$} (m-2-2)
            ;
        \end{tikzpicture}
        \]
        Recall that the map $\phi : \hat{\alpha}(\ast) \to_\sbt \hat{\alpha}'(\ast)$ corresponds to a group morphism $\psi : H \to_\Grp H'$.
        Then, \cref{thm4} implies that the above square is pointed commutative exactly when
        the square of group morphisms
        \[
        \begin{tikzpicture}
          \matrix (m) [matrix of math nodes,row sep=2em,column sep=2em,minimum width=2em,nodes={text height=1.75ex,text depth=0.25ex}]
            { H & H' \\
              H & H' \\};
          \path[-stealth]
            (m-1-1) edge node [above] {$\psi$} (m-1-2)
                    edge node [left] {$\alpha(|l|)$} (m-2-1)
            (m-2-1) edge node [above] {$\psi$} (m-2-2)
            (m-1-2) edge node [right]{$\alpha'(|l|)$} (m-2-2)
            ;
        \end{tikzpicture}
        \]
        is commutative.
        So the type $\prd{x : X'} \hat{\alpha}(x) \to_\sbt \hat{\alpha}'(x)$ is equivalent to the type of group morphisms
        $H \to_\Grp H'$ that respect the actions, as needed.
    \end{proof}

\begin{defn}
    A map $i : H \to_\Grp H'$ is a \define{normal inclusion} if it is injective and whenever we have
    $h : H$, $h' : H'$, the type $\fib{i}{h' i(h) h'^{-1}}$ is inhabited.
    We denote the type of normal inclusions of $H$ into $H'$ by $H \lhd H'$.
\end{defn}

Notice that being a normal inclusion is a mere property of a group morphism.

\begin{eg}
    Any inclusion with codomain an abelian group is a normal inclusion.
\end{eg}

Using quotients it is easy to construct the cokernel, or quotient, of a normal inclusion.
We will not give the details here, since the construction is analogous to the set theoretic construction.

\begin{defn}
    Given a normal inclusion $i : H \to_\Grp H'$ we denote its \define{cokernel}, or \define{quotient},
    by $\coKer(i) : \Grp$, or $H'/H : \Grp$.
\end{defn}

\begin{defn}
    Given an action $\alpha' : G \curvearrowright H'$, a \define{normal subaction} is given by an action
    $\alpha : G \curvearrowright H$, and a map of actions $\phi : \alpha \to_\act \alpha'$
    such that its corresponding group morphism $H \to_\Grp H'$ is a normal inclusion.
    The type $\alpha \lhd_\act \alpha'$ denotes the type of maps of actions $\alpha \to_\act \alpha'$, together
    with a proof that the underlying morphism is a normal inclusion.
\end{defn}

In the context of the above definition, the usual set theoretic construction provides,
for each normal subaction $\phi : \alpha \to_\act \alpha'$, a quotient action $\alpha'/\alpha : G \curvearrowright H'/H$.

\begin{defn}\label{remark:actionofses}
For a group morphism $G \xrightarrow{q} H$,
let the \define{kernel} of $q$ be the group morphism $I \to_\Grp G$ where
$I \defeq \left(\sum_{(l : G)} q(l) = e\right)$ with the group structure inherited from $G$,
and the morphism $I \to_\Grp G$ given by the first projection.

There is an induced action $G \curvearrowright I$ defined as follows.
Given $g : G$ and $(l, r) : I$, let
$g\cdot (l,r) \defeq (g^{-1} l g, \phi_{r})$, where $\phi_{r}$ is the straightforward proof that
$g^{-1} l g$ gets mapped to $e : H$ by $q$.
It is easy to check that this action is a normal subaction of the action $G\curvearrowright G$
given by conjugation.
\end{defn}

\begin{defn}
A sequence of group morphisms $I \to G \to H$ is \define{short exact} if $G \to H$ is surjective
and $I \to G$ is equivalent to the kernel of $G \to H$.
    A short exact sequence of groups $I \to G \to H$ is a \define{central extension}
    if the action $G \curvearrowright I$ given in \cref{remark:actionofses} is trivial.
\end{defn}

Note that being short exact and being a central extension are both mere properties
of a sequence $I \to G \to H$.

As is standard in group theory, whenever we have normal inclusions $H \lhd H' \lhd G$
such that the composite inclusion $H \to_\Grp G$ is normal, we can form a short exact sequence
of groups $H'/H \to G/H \to G/H'$.

\begin{defn}
    Let $H$ be a group. Given a sequence of group morphisms
    \[
        \unit \equiv H_0 \to H_1 \to \cdots \to H_k \equiv H,
    \]
    we say that it is a \define{central series} if all the morphisms are normal inclusions,
    all the composites $H_i \to H$ are normal inclusions,
    and for all $0 \leq i < k$, the short exact sequence
    \[
        H_{i+1}/H_i \to H/H_i \to H/H_{i+1}
    \]
    is a central extension.
\end{defn}

Note that being a central series is a mere property.

\begin{defn}
    The type of \define{nilpotent structures} for a group $H$ is defined to be the type of central series for $H$.
\end{defn}

\begin{defn}
    Given an action $\alpha : G\curvearrowright H$ of a group on another group,
    a \define{nilpotent structure} on this action consists of a sequence of maps of actions
    \[
        \unit \equiv \alpha_0 \to_\act \alpha_1 \to_\act \cdots \to_\act \alpha_k \equiv \alpha,
    \]
    such that the underlying sequence of group morphisms is a central series for $H$, and such that,
    for every $i$, the map $\alpha_i \to_\act \alpha_{i+1}$ and
    the composite $\alpha_i \to_\act \alpha$ are normal subactions, and
    the quotient action $\alpha_{i+1}/\alpha_i : G \curvearrowright H_{i+1}/H_i$ is trivial.
\end{defn}

Notice that the type of nilpotent structures on a group $G$ is equivalent to the type of nilpotent
structures for the action $G \curvearrowright G$ given by conjugation.

\begin{lem}\label{lemma:injectionssurjections}
    The type of nilpotent structures on a group $G$ is equivalent to the type of finite sequences of group epimorphisms
    \[
        G \equiv G_0' \twoheadrightarrow G_1' \twoheadrightarrow \cdots \twoheadrightarrow G_k' \equiv \unit
    \]
    such that for every $0 \leq i < k$, the short exact sequence
    \[
        K_i \to G_i' \twoheadrightarrow G_{i+1}',
    \]
    where $K_i$ is the kernel, is a central extension.
    The equivalence is given by mapping a nilpotent structure for $G$
    \[
        \unit \equiv G_0 \lhd G_1 \lhd \cdots \lhd G_k \equiv G
    \]
    to the sequence of group epimorphisms given by setting $G_i' \defeq G/G_i$, and defining the morphism $G_i' \to_\Grp G_{i+1}'$ to be the epimorphism
    $G/G_i \twoheadrightarrow G/G_{i+1}$.
\end{lem}
\begin{proof}
    It is clear that the short exact sequence
    \[
        G_{i+1}/G_i \to G/G_i \to G/G_{i+1}.
    \]
    is isomorphic to the short exact sequence
    \[
        K_i \to G_i' \to G_{i+1}',
    \]
    and thus this last short exact sequence is a central extension, as needed.

    Going the other way, assume given 
    \[
        G \equiv G_0' \twoheadrightarrow G_1' \twoheadrightarrow \cdots \twoheadrightarrow G_k' \equiv \unit
    \]
    such that the short exact sequence
    \[
        K_i \to G_i' \to G_{i+1}'
    \]
    is a central extension for every $i$.
    Let $G_i$ be the kernel of the map $G \to G'_i$ obtained by composing the maps in the sequence.
    By functoriality of kernels, we have induced maps
    \[
        1 \equiv G_0 \to G_1 \to \cdots \to G_k \equiv G,
    \]
    which give rise to a filtration of $G$.
    We will prove that this filtration is a nilpotent structure for $G$.
    Since $G_i$ is normal in $G$, it must be normal in $G_{i+1}$, so the filtration consists of normal inclusions.
    It remains to show that the short exact sequences
    \[
        G_{i+1}/G_i \to G/G_i \to G/G_{i+1}
    \]
    are central extensions.
    By construction, this sequence is isomorphic to
    \[
        K_i \to G'_i \to G'_{i+1}
    \]
    which is a central extension by hypothesis.

    The fact that the constructions form an equivalence is clear.
\end{proof}

\subsection{$K(-,n)$-bundles}
\label{Kbundles}

    The goal of this section is to give a hands-on description of fibrations with fiber an Eilenberg--Mac Lane space
    (\cref{remark:descriptionbundles}) and their associated actions
    that will allow us to work with principal fibrations and nilpotent types.
    This description follows from \cref{proposition:characterizationF}, which,
    inspired by \citep{shulman}, studies the classifying spaces of $K(A,n)$-bundles, and establishes the
    equivalence between the type of unpointed $n$-dimensional
    Eilenberg--Mac Lane spaces and the type of doubly-pointed $(n+1)$-dimensional Eilenberg--Mac Lane spaces.

    The main reason to introduce doubly-pointed Eilenberg--Mac Lane spaces is the following.
    In order to associate an action of $\pi_1(X)$ on a group $G$ to a $K(G,n)$-bundle over $X$ in a natural way,
    we must construct a map from the space classifying $K(G,n)$-bundles to the space classifying $\pi_1(X)$-actions
    on $G$. To state this more concretely, we give the following definition.

\begin{defn}\label{unpointedEM}
Given a natural number $n \geq 1$, and a group $G$ (assumed to be abelian if $n > 1$),
define the \define{type of unpointed Eilenberg--Mac Lane spaces} of type $G$, $n$
as follows:
\[
    \EM(G,n) \defeq \sum_{K : \UU} \ttrunc{-1}{K =_{\UU} K(G,n)}.
\]
\end{defn}

    We show in \cref{remark:classificationanduniqueness}
    that the type of \cref{unpointedEM} classifies $K(G,n)$-bundles,
    so we need to construct a map $\EM(G,n) \to \EM_\sbt(G,n)$,
    since $\EM_\sbt(G,n)$ classifies $\pi_1(X)$-actions on $G$ (\cref{remark:equivalenceofEM}).
    In \citep{shulman}, this is done by showing that the map $\EM_\sbt(G,n) \to \EM(G,n)$ that forgets
    the pointing admits a retraction. The argument uses the truncated Whitehead theorem, and thus the corresponding 
    retraction is hard to work with. In this section we give a more explicit characterization of the retraction by showing
    that the map $\EM_\sbt(G,n) \to \EM(G,n)$ is equivalent to the map
    $\EM_\sbt(G,n+1) \to \EM_{\sbt\sbt}(G,n+1)$ that repeats the pointing, and as retraction we use the
    map $\EM_{\sbt\sbt}(G,n+1) \to \EM_\sbt(G,n+1)$ that forgets one of the points.
    Here, $\EM_{\sbt\sbt}(G,n) \defeq \sum_{K : \UU_{\sbt\sbt}} \ttrunc{-1}{K =_{\UU_{\sbt\sbt}} K(G,n)}$
    is the space of doubly-pointed Eilenberg--Mac Lane spaces of type $G$, $n$.

    In this section, $A$ and $B$ will denote abelian groups, and $G$ an arbitrary group.
    As usual, when writing $K(G,n)$, we will assume that $G$ is abelian, if $n>1$.

    \begin{defn}
        Given $n\geq 1$, a \define{$K(G,n)$-bundle} is a map such that all of its fibers are merely equivalent to $K(G,n)$.
        Formally, for a map $f : Y \to X$, we let
        \[
            \isKGnbun(f) \defeq \prod_{x : X} \ttrunc{-1}{\fib{f}{x} \simeq K(G,n)}.
        \]
        A \define{pointed $K(G,n)$-bundle} is a pointed map that is also a $K(G,n)$-bundle,
        together with a proof that its fiber over the base point is equivalent to $K(G,n)$ as a pointed type.
        Formally, for a pointed map $f : Y \to_\sbt X$, with $x_0 : X$ the basepoint of $X$, we let
        \[
            \isKGnbun_\sbt(f) \defeq \isKGnbun(f) \times \left(\fib{f}{x_0} \simeq_\sbt K(G,n)\right).
        \]
    \end{defn}
    Equivalently, a pointed $K(G,n)$-bundle is a pointed map $f : Y \to_\sbt X$ whose underlying (unpointed) map
    is a $K(G,n)$-bundle, together with a pointed map $K(A,n) \to_\sbt Y$, and a proof that the following square
    commutes as a square of pointed maps, and is a pullback:
     \[
        \begin{tikzcd}
            K(A,n) \ar[d]\ar[r] & Y \ar[d] \\
            \unit \ar[r] & X.
        \end{tikzcd}
    \]   

    We will write $K(-, n)$-bundle when we mean a $K(G, n)$-bundle for some group $G$.

    The forgetful map $F : \EM_\sbt(G,n) \to \EM(G,n)$ classifies $K(G,n)$-bundles
    and pointed $K(G,n)$-bundles in the following sense.

    \begin{lem}\label{remark:classificationanduniqueness}
    Let $X : \UU$. There is an equivalence between maps $X \to \EM(G,n)$
    and $K(G,n)$-bundles with codomain $X$.
    The equivalence is given, on the one hand, by pulling back $F$ along the map $X \to \EM(G,n)$
    to get a map $Y \to X$. And on the other hand, by sending $f : Y \to X$ to the type family $\fib{f}{-} : X \to \EM(G,n)$.

    In particular, the map
    \begin{align*}
        Y &\to \sum_{x:X} \fib{f}{x}\\
        y &\mapsto (f(y), (y,\refl{}))
    \end{align*}
    is an equivalence over $X$. This equivalence is natural in $Y$, in that,
    given $K(G,n)$-bundles $f : Y \to X$ and $f' : Y' \to X$, a map $g : Y \to Y'$, and
    a homotopy $h : f' \circ g \sim f$, we get an induced fiberwise map
    \begin{align*}
        \prod_{x : X} \fib{f}{x} &\xrightarrow{g_{\fibb{}}} \fib{f'}{x}\\
                         (y,p)   &\mapsto (g(y), h(y) \sq p)
    \end{align*}
    and the homotopy that is constantly reflexivity makes the following square commute
    \[
        \begin{tikzpicture}
          \matrix (m) [matrix of math nodes,row sep=2em,column sep=6em,minimum width=2em,nodes={text height=1.75ex,text depth=0.25ex}]
            { Y  & Y' \\
              \sum_{x:X}\fib{f}{x}  & \sum_{x:X}\fib{f'}{x} .  \\};
          \path[-stealth]
            (m-1-1) edge node [above] {$g$} (m-1-2)
                    edge node [left] {$\sim$} (m-2-1)
            (m-1-2) edge node [right] {$\sim$} (m-2-2)
            (m-2-1) edge node [above] {$(\idfunc{}, g_{\fibb{}})$} (m-2-2)
            ;
        \end{tikzpicture}
    \]
    \end{lem}
    \begin{proof}
        This is a specialization of \cite[Theorem~4.8.3]{hottbook}, where instead of the full
        universe $\UU$, we have the subuniverse of types equivalent to $K(G,n)$, namely
        $\EM(G,n)$.
    \end{proof}

    The same results in \cref{remark:classificationanduniqueness}
    hold when working with pointed $K(G,n)$-bundles:
    in this case we see $F$ as a pointed map
    equipping the domain and codomain with the point given by $K(G,n)$.
    In that case, we have an equivalence between pointed maps $X \to_\sbt \EM(G,n)$ and
    pointed $K(G,n)$-bundles with codomain $X$, defined again by pulling back $F$ along
    the pointed map $X \to_\sbt \EM(G,n)$.

    Notice that given a map $X \to K(A,n+1)$, its fiber $Y\to X$ is a $K(A,n)$-bundle.
    But, as we will see in \cref{proposition:characterizationprincipalfibration},
    not every $K(A,n)$-bundle arises in this way.
    
\begin{defn}
    A (pointed) $K(A,n)$-bundle $f : Y \to X$ is a \textbf{(pointed) principal fibration}
    if it is the fiber of a (pointed) map $f' : X \to K(A,n+1)$.
    Formally we define
    \[
        \isPrincipal(f) \defeq \sum_{\substack{f' \,:\, X \to K(A,n+1)\\e \,:\, \fib{f'}{\ast} \,\simeq\, Y\,\,\,\,\,\,}} f \circ e \sim i,
    \]
    where $i : \fib{f'}{\ast} \to X$ is the projection from the fiber of $f'$ to $X$.
    In the case of pointed principal fibrations, $f'$, $e$, and the homotopy $f \circ e \sim i$ are taken to be pointed.
\end{defn}

    The goal of this section is to characterize principal fibrations and, more generally, maps
    that can be factored as a finite composite of principal fibrations.

    We start with a stabilization theorem for Eilenberg--Mac Lane spaces of abelian groups that
    follows directly from the main construction in \citep{FinsterLicata}.

\begin{lem}\label{lemma:freudKGn}
    Given $n\geq 1$ and $(K,k) : \EM_\sbt(A,n)$, the following composite is a pointed equivalence:
    \[
        K \to \loopspacesym \suspsym K \to \loopspacesym \ttrunc{n+1}{\suspsym K}.
    \]
    The first map is the Freudenthal map, and the second one is the functorial action of $\loopspacesym$ on
    the $(n+1)$-truncation unit. We are equipping suspensions with a point using the point constructor $\north$.
\end{lem}
\begin{proof}
    To see that it is a pointed map notice that the base point $k : K$ gets mapped to $\refl{\north}$ by the Freudenthal map.
    The second map then sends this point to $\refl{|\north|}$, by the functoriality of $\loopspacesym$ on paths.

    Since being an equivalence is a mere proposition, it is enough to prove the statement when $K \equiv K(A,n)$.
    Although this is proven in \citep[Theorem~5.4]{FinsterLicata}, it is not stated in exactly this way, so we give some details.
    In the proof, the cases $n=1$ and $n>1$ are considered separately.
    When $n=1$, \citep[Theorem~4.3]{FinsterLicata} is used, and the Freudenthal map appears as $\decode'$.
    When $n>1$, \citep[Lemma~5.3]{FinsterLicata} is used, and in that case the Freudenthal map appears explicitly
    in \citep[Corollary~5.2]{FinsterLicata}.
\end{proof}

    The stabilization theorem can be used to construct an equivalence $\EM_\sbt(A,n) \simeq \EM_\sbt(A,n+1)$,
    as in \citep[Theorem~6.7]{BRV}.
    The next theorem (\cref{proposition:characterizationF}) is a generalization of this fact. We start with the crucial construction.

    Consider the maps $\sigma : \EM_\sbt(A,n) \rightleftarrows \EM_\sbt(A,n+1) : \rho$
    given by 
    \begin{align*}
        \sigma(K,k) &\defeq \left(\ttrunc{n+1}{\susp K}, |\north|\right),\\
        \rho(K,k) &\defeq \left(\loopspacesym(K,k),\refl{k}\right).
    \end{align*}
    Here $\north$ and $\south$ are the point constructors of the suspension.
    Notice that we are omitting the proofs that the maps land in the correct types,
    which follows from \cref{lemma:freudKGn}.

    Consider also the ``non-pointed versions'' of $\sigma$ and $\rho$, namely
    the maps $\s : \EM(A,n) \rightleftarrows \EM_{\sbt\sbt}(A,n+1) : \rr$ given by
    \begin{align*}
        \s(K) &\defeq \left(\ttrunc{n+1}{\suspsym K}, |\north|, |\south|\right),\\
        \rr(K',p,q) &\defeq (p =_{K'} q).
    \end{align*}

    We have the following commutative square:
    \[
        \begin{tikzpicture}
          \matrix (m) [matrix of math nodes,row sep=2em,column sep=2em,minimum width=2em,nodes={text height=1.75ex,text depth=0.25ex}]
          { \EM_\sbt(A,n) & \EM_\sbt(A,n+1) \\
            \EM(A,n) & \EM_{\sbt\sbt}(A,n+1), \\};
          \path[-stealth]
            (m-1-1) edge node [above] {$\sigma$} (m-1-2)
                    edge node [left] {$F$} (m-2-1)
            (m-2-1) edge node [above] {$\s$} (m-2-2)
            (m-1-2) edge node [right]{$R$} (m-2-2)
            ;
        \end{tikzpicture}
    \]
    where $R : \EM_\sbt(A,n+1) \to \EM_{\sbt\sbt}(A,n+1)$ is the map that repeats the pointing.
    To see that this square commutes notice that given $(K,k) : \EM_{\sbt}(A,n)$,
    the carriers of $\s(F(K,k))$ and $R(\sigma(K,k))$ are definitionally equal.
    To see that the (double) pointings coincide, notice that $\s(F(K,k))$ is pointed 
    with $|\north|$ and $|\south|$, whereas $R(\sigma(K,k))$ is pointed with $|\north|$ and $|\north|$.
    But since we have $k : K$, we get $\ap{|-|}(\merid(k)) : |\north| = |\south|$, as required.

    Similarly, we have the following square that commutes on the nose:
    \[
        \begin{tikzpicture}
          \matrix (m) [matrix of math nodes,row sep=2em,column sep=2em,minimum width=2em,nodes={text height=1.75ex,text depth=0.25ex}]
          { \EM_\sbt(A,n) & \EM_\sbt(A,n+1) \\
            \EM(A,n) & \EM_{\sbt\sbt}(A,n+1). \\};
          \path[-stealth]
            (m-1-2) edge node [above] {$\rho$} (m-1-1)
            (m-1-1) edge node [left] {$F$} (m-2-1)
            (m-2-2) edge node [above] {$\rr$} (m-2-1)
            (m-1-2) edge node [right]{$R$} (m-2-2)
            ;
        \end{tikzpicture}
    \]

\begin{thm}\label{proposition:characterizationF}
    For any $n \geq 1$, the pair $(\sigma, \s)$ gives an equivalence between
    the map $F : \EM_\sbt(A,n) \to \EM(A,n)$ and the map $R : \EM_\sbt(A,n+1) \to \EM_{\sbt\sbt}(A,n+1)$.
    Its inverse is the pair $(\rho, \rr)$.
\end{thm}
\begin{proof}
    Our goal is to show that the pairs $\sigma$ and $\rho$, and $\s$ and $\rr$ form equivalences.
    The arguments in this proof are easier to understand by informally thinking of $F$, $R$, $\sigma$, $\rho$, $\s$ and $\rr$
    as $(\infty,1)$-functors between $(\infty,1)$-categories.
    From this perspective, the proof strategy is to show that $\sigma$ and $\rho$, and $\s$ and $\rr$ form ``adjoint equivalences''.
    Start by noticing that we have maps (which we interpret as unit and counit maps for $\sigma$ and $\rho$)
    \begin{align*}
        \mu &: (K,k) \to_\sbt \left( \loopspacesym(\ttrunc{n+1}{\suspsym K},|\north|), \refl{\north} \right) \equiv \rho(\sigma(K,k)),\\
        \epsilon &: \sigma(\rho(K,k)) \equiv \left(\ttrunc{n+1}{\suspsym \loopspacesym (K,k)}, |\north|\right) \to_{\sbt} (K,k),
    \end{align*}
    defined using the map of \cref{lemma:freudKGn} for $\mu$ and by truncation-induction and suspension-induction,
    mapping $|\north|$, $|\south|$, and $\ap{|-|}(\merid(e : k = k))$ to $k$, $k$, and $e$, respectively.
    Similarly, we have maps (which we interpret as unit and counit maps for $\s$ and $\rr$)
    \begin{align*}
        u &: K \to \left(|\north| =_{\ttrunc{n+1}{\suspsym K}} |\south|\right) \equiv \rr(\s(K)),\\
        c &: \s(\rr(K',p,q)) \equiv \left(\ttrunc{n+1}{\suspsym (p =_{K'} q)}, |\north|, |\south|\right) \to_{\sbt\sbt} (K', p, q),
    \end{align*}
    with $u$ defined by $u(x) \defeq \ap{|-|}(\merid(x))$,
    and $c$ defined by truncation-induction and suspension-induction, mapping $\north$ to $p$, $\south$ to $q$, and an equality $e : p = q$ to itself.
    It is then enough to show that $\mu$, $\epsilon$, $u$, and $c$ are equivalences,
    since, together with univalence, this establishes the result.

    \cref{lemma:freudKGn} directly implies that $\mu$ is an equivalence.
    The proof that $\epsilon$ is an equivalence requires a bit more work.
    Assume given an $n$-connected and $(n+1)$-truncated type $K$, and $k : K$.
    Since the domain and codomain of the counit are connected, it is enough to show that
    $\Omega \epsilon$
    is an equivalence (using \citep[Theorem~8.8.2]{hottbook}).
    To do this, we must prove and use one of the ``triangle identities''. Concretely, we want to show that the following
    triangle commutes:
    \[
        \begin{tikzpicture}
          \matrix (m) [matrix of math nodes,row sep=2em,column sep=4em,minimum width=2em,nodes={text height=1.75ex,text depth=0.25ex}]
            { \loopspacesym(K,k) & \loopspacesym\left(\ttrunc{n+1}{\suspsym \loopspacesym(K,k)}\right) \\
                  & \loopspacesym(K,k).  \\};
          \path[-stealth]
            (m-1-1) edge node [above] {$\mu_{\loopspacesym(K,k)}$} (m-1-2)
                    edge [double equal sign distance,-] node [left] {} (m-2-2)
            (m-1-2) edge node [right] {$\loopspacesym \epsilon_{(K,k)}$} (m-2-2)
            ;
        \end{tikzpicture}
    \]
    The top map is homotopic to the map that sends $l : k = k$ to $\ap{|-|}(\merid(l))\sq\ap{|-|}(\merid(\refl{k}))$.
    And since $\loopspacesym \epsilon$ respects composition, we see that $\ap{|-|}(\merid(l))\sq\ap{|-|}(\merid(\refl{k}))$
    gets mapped to $l \sq \refl{k} = l$, by the computation rules of truncation-induction and suspension-induction.
    This establishes the commutativity of the triangle.
    To use the ``triangle identity'', just notice that the top map is an equivalence and thus the right vertical map must be an equivalence, as needed.

    To prove that $u$ and $c$ are equivalences, we will use the fact that $F$ and $R$ respect the units and counits.
    By that we mean that, for $(K,k) : \EM_{\sbt}(A,n)$, we have a commutative square
    \[
        \begin{tikzpicture}
          \matrix (m) [matrix of math nodes,row sep=2em,column sep=2em,minimum width=2em,nodes={text height=1.75ex,text depth=0.25ex}]
            { F(K,k) & F(\rho(\sigma(K,k))) \\
              F(K,k) & \rr(\s(F(K,k))) \\};
          \path[-stealth]
            (m-1-1) edge node [above] {$F\mu$} (m-1-2)
                    edge [double equal sign distance,-] node [left] {} (m-2-1)
            (m-2-1) edge node [above] {$u$} (m-2-2)
            (m-1-2) edge node [right]{$\sim$} (m-2-2)
            ;
        \end{tikzpicture}
    \]
    and for $(K,k) : \EM_{\sbt}(A,n+1)$ we have a commutative square
    \[
        \begin{tikzpicture}
          \matrix (m) [matrix of math nodes,row sep=2em,column sep=2em,minimum width=2em,nodes={text height=1.75ex,text depth=0.25ex}]
            { R(\sigma(\rho(K,k))) & R(K,k) \\
              \s(\rr(R(K,k))) & R(K,k). \\};
          \path[-stealth]
            (m-1-1) edge node [above] {$R\epsilon$} (m-1-2)
                    edge node [left] {$\sim$} (m-2-1)
            (m-2-1) edge node [above] {$c$} (m-2-2)
            (m-1-2) edge [double equal sign distance,-] node [right]{} (m-2-2)
            ;
        \end{tikzpicture}
    \]

    To construct the first square we let the right vertical equivalence
    \[
        F(\rho(\sigma(K,k))) \equiv \left(|\north| =_{\ttrunc{n+1}{\suspsym K}} |\north|\right)
        \simeq \left(|\north| =_{\ttrunc{n+1}{\suspsym K}} |\south|\right) \equiv \rr(\s(F(K,k)))
    \]
    be given by concatenation with $\ap{|-|}(\merid(k))$.
    It is then clear that the square commutes, since the top map is
    homotopic to $x \mapsto \ap{|-|}(\merid(x))\sq\ap{|-|}(\merid(k)^{-1})$
    by the functoriality of $\ap{}$.

    The equivalence for the second square
    \[
        R(\sigma(\rho(K,k))) \equiv \left(\ttrunc{n+1}{\suspsym (k = k)}, |\north|, |\north|\right) \simeq
        \left(\ttrunc{n+1}{\suspsym (k = k)}, |\north|, |\south|\right) \equiv \s(\rr(R(K,k)))
    \]
    is given by the identity on the carrier, and by equating $|\north|$ and $|\south|$ using $\ap{|-|}(\merid(\refl{k}))$ for
    the second pointing.
    To see that the square commutes start by noticing that the composites agree on the carriers definitionally.
    The top right composite is pointed using $\refl{k}$ for both points.
    The left bottom composite is pointed using $\refl{k}$ for the first point and
    a path homotopic to $\refl{k}$ for the second point, by the computation rule of truncation-induction and
    suspension-induction. So the square in fact commutes.

    Finally, the commutativity of the first square implies that $u$ is an equivalence whenever we have $k : K$.
    But since being an equivalence is a mere proposition and any $K : \EM(A,n)$ is
    merely inhabited, $u$ is always an equivalence.
    Similarly, the commutativity of the second square implies that $c$ is also an equivalence.
    \end{proof}

    It should be possible to generalize the above argument to encompass bundles with fiber any higher group
    in a stable range, as in \citep[Section~6]{BRV}.

    Since $K(A,n)$-bundles over $X$ are classified by maps $X \to \EM_\sbt(A,n)$, we can use the above properties
    for each $x : X$ to get an alternative description of $K(-,n)$-bundles and of fiberwise maps between them.

    Recall the map $\s : \EM(A,n) \to \EM_{\sbt\sbt}(A,n+1)$ given by suspension followed by $(n+1)$-truncation.

    \begin{defn}
        Given a (pointed) $K(A,n)$-bundle $f : Y \to X$, its \define{classifying map}
        is defined to be the corresponding map $\ceil{f} \defeq \s\circ\fib{f}{-} : X \to \EM_{\sbt\sbt}(A,n+1)$.
    \end{defn}
   
    \begin{cor}\label{corollary:classification}
        The type of (pointed) $K(A,n)$-bundles over $X$ is equivalent to the type of (pointed) maps $X \to \EM_{\sbt\sbt}(A,n+1)$.
        This equivalence is given by mapping a $K(A,n)$-bundle $f : Y \to X$ to its classifying map $\ceil{f} : X \to \EM_{\sbt\sbt}(A,n+1)$.\qed
    \end{cor}

    The above result says that a (pointed) $K(A,n)$-bundle over $X$ is determined by the following data:
        \begin{itemize}
            \item a (pointed) map $\overline{f} : X \to \EM(A,n+1)$;
            \item two (pointed) sections $s_1, s_2 : \prd{x : X} \overline{f}(x)$.
        \end{itemize}
    More concretely, we have the following.

    \begin{cor}\label{remark:descriptionbundles}
    Let $X : \UU$.
    Any $K(A,n)$-bundle $f : Y \to X$ is equivalent to the projection map $pr_1 : \sum_{x:X} s_1(x) =_{\overline{f}(x)} s_2(x) \to X$,
   where $(\overline{f},s_1,s_2) \equiv \ceil{f}$ as in \cref{corollary:classification}, and the fiberwise equivalence $Y \to \sum_{x:X} s_1(x) =_{\overline{f}(x)} s_2(x)$ is given by the composite
        \[
            Y \to \sum_{x:X}\fib{f}{x} \xrightarrow{(\idfunc{},u)} \sum_{x:X} |\north| =_{\ttrunc{n+1}{\fib{f}{x}}} |\south|.
        \]
    \end{cor}
    \begin{proof}
        This follows from the equivalences of \cref{proposition:characterizationF} and \cref{remark:classificationanduniqueness}.
    \end{proof}

    As was hinted in its proof, the equivalence of \cref{proposition:characterizationF}
    is not just an equivalence of types.
    Although more general properties can be proven, it will be enough for us to know the following.

    \begin{prop}\label{proposition:functoriality}
    Given $K : \EM(A,n)$, $K' : \EM(B,n)$, and a map $\phi : K \to K'$, consider the map
    $\phi' \defeq \ttrunc{n+1}{\suspsym \phi} : \s(K) \equiv \ttrunc{n+1}{\suspsym K} \to \ttrunc{n+1}{\suspsym K'} \equiv \s(K')$
    defined by truncation-induction and
    suspension-induction, by mapping $\north, \south : \suspsym K$ to $\north, \south : \suspsym K'$ respectively,
    and $\merid(k)$ to $\merid(\phi(k))$.
    Then, the computation rule of suspension-induction gives a homotopy making the following square commute
    \[
        \begin{tikzpicture}
          \matrix (m) [matrix of math nodes,row sep=2em,column sep=2em,minimum width=2em,nodes={text height=1.75ex,text depth=0.25ex}]
            { K & K' \\
              \rr(\s(K)) & \rr(\s(K')). \\};
          \path[-stealth]
            (m-1-1) edge node [above] {$\phi$} (m-1-2)
                    edge node [left] {$u$} (m-2-1)
            (m-2-1) edge node [above] {$\ap{\phi'}$} (m-2-2)
            (m-1-2) edge node [right] {$u$} (m-2-2)
            ;
        \end{tikzpicture}
    \]
    \end{prop}
    \begin{proof}
        Given $k : K$, we have on the one hand $u(\phi(k)) \equiv \ap{|-|}(\merid(\phi(k)))$.
        On the other hand $\ap{\phi'}(u(k)) \equiv \ap{\phi'}(\ap{|-|}(\merid(k))) = \ap{|-|}(\merid(\phi(k)))$,
        where the equality is given by the computation rule of suspension-induction.
    \end{proof}

    \cref{proposition:functoriality} implies that the description in \cref{remark:descriptionbundles} is functorial.
    To see this, assume given
    $K(-,n)$-bundles $f : Y \to X$, $f' : Y' \to X$ and a map $g : Y \to Y'$, such that $f' \circ g \sim f$.
    We have $t_g : \prod_{x : X} \overline{f}(x) \to \overline{f'}(x)$
    given by $t_g(x) \equiv (g_{\fibb{}}(x))'$, as in \cref{proposition:functoriality}.
    For every $x : X$ the map $t_g(x)$ is doubly-pointed definitionally, since it is defined
    by suspension-induction.
    Then \cref{proposition:functoriality} implies the following.

\begin{cor}\label{proposition:functorialitycharacterizationbundles}
    Assume given $K(-,n)$-bundles $f : Y \to X$, $f' : Y' \to X$ and a map $g : Y \to Y'$, such that $f' \circ g \sim f$.
    Then, the computation rule of suspension-induction gives a homotopy making the following square commute:
    \[
        \begin{tikzpicture}
          \matrix (m) [matrix of math nodes,row sep=2em,column sep=6em,minimum width=2em,nodes={text height=1.75ex,text depth=0.25ex}]
            { Y & Y' \\
              \left(\sm{x :X} s_1(x) =_{\overline{f}(x)} s_2(x)\right)  & \left(\sm{x :X} s'_1(x) =_{\overline{f'}(x)} s'_2(x)\right).  \\};
          \path[-stealth]
            (m-1-1) edge node [above] {$g$} (m-1-2)
                    edge node [left] {$\sim$} (m-2-1)
            (m-1-2) edge node [right] {$\sim$} (m-2-2)
            (m-2-1) edge node [above] {$(\idfunc{}, \lambda x.\ap{t_g(x)})$} (m-2-2)
            ;
        \end{tikzpicture}
    \]
    Here the vertical maps are given by the equivalence of \cref{remark:descriptionbundles}.\qed
\end{cor}

    Let $F_\sbt : \EM_{\sbt\sbt}(A,n+1) \to \EM_{\sbt}(A,n+1)$ be the map that forgets the second point.

    \begin{defn}\label{definition:associatedaction}
        Given a pointed $K(A,n)$-bundle $f : Y \to X$ we define the \define{associated action}
        of $X$ on $A$ as the pointed map $F_\sbt \circ \ceil{f} : X \to_\sbt \EM_\sbt(A,n+1)$.

        We say that a pointed $K(A,n)$-bundle $f$ \define{lives over} an action $\phi : X \to_\sbt \EM_\sbt(A,n+1)$ if
        $\phi$ is equal to the action associated to $f$.
    \end{defn}

    \begin{eg}
        As is familiar from Homotopy Theory, for any pointed type $X$ and $n\geq 1$,
        we have an action $\pi_1(X) \curvearrowright \pi_n(X)$.
        The action of $\pi_1(X)$ on itself is defined to be the action by conjugation.
        The action of $\pi_1(X)$ on $\pi_n(X)$ is defined to be the action associated to the fibration $\ttrunc{n}{X} \to \ttrunc{n-1}{X}$,
        using \cref{definition:associatedaction} and \cref{lemma:actionisaction}.
    \end{eg}

    \begin{rmk}
    In the same vein as \cref{remark:descriptionbundles}, the associated action of a pointed $K(A,n)$-bundle over $X$
    is nothing but:
    \begin{itemize}
        \item the pointed map $\overline{f} : X \to_\sbt \EM(A,n+1)$;
        \item the first of its pointed sections $s_1 : \prod^\sbt_{x:X}\overline{f}(x)$.
    \end{itemize}
        where $\prod^\sbt$ is the pointed pi-type. An element of the pointed pi-type $\prod^\sbt_{x:X}\overline{f}(x)$
        consists of an element $s$ of the usual, unpointed
        pi-type, together with a path from $s(x_0)$ to the base point of $\overline{f}(x_0)$.
    \end{rmk}

    \begin{lem}\label{lemma:functorialityaction}
        Given pointed $K(-,n)$-bundles $f : Y \to_\sbt X$ and $f' : Y' \to_\sbt X$, and a pointed map
        $g : Y \to_\sbt Y'$ such that $f = f' \circ g$ as pointed maps, we have an induced map of actions
        $(\overline{f},s_1) \to_\act (\overline{f}',s_1')$.
    \end{lem}
    \begin{proof}
        Using the description of \cref{proposition:functorialitycharacterizationbundles},
        we see that we have in particular a map of actions $(\overline{f},s_1) \to_\act (\overline{f}',s_1')$.
    \end{proof}

\subsection{Principal fibrations}
\label{principalfibrations}
In this section, following \citep{shulman}, we prove that
principal fibrations are precisely the fibrations that live over the trivial action.
We will need the following general lemma.

    \begin{lem}
        \label{lemma:trivialintersection}
        Assume given a pair of maps $r : E \leftrightarrows B : s$ such that $q : r\circ s \sim \idfunc{B}$.
        Then, for every $b : B$, the following square is a pullback
        \begin{center}
            \begin{tikzpicture}
              \matrix (m) [matrix of math nodes,row sep=2em,column sep=2em,minimum width=1em,nodes={text height=1.75ex,text depth=0.25ex}]
                { \unit & B \\
                \fib{r}{b} & E. \\};
              \path[-stealth]
                (m-1-1) edge node [above] {$b$} (m-1-2)
                        edge node [left] {$(s(b),q(b))$} (m-2-1)
                (m-2-1) edge node [above] {} (m-2-2)
                (m-1-2) edge node [right]{$s$} (m-2-2)
                ;
            \end{tikzpicture}
        \end{center}
    \end{lem}
    \begin{proof}
        This can be proven by writing the pullback using sigma types, or by using the two pullback lemma
        \citep[Lemma~4.1.11]{AKL} applied to the following diagram:
        \[
            \begin{tikzpicture}
              \matrix (m) [matrix of math nodes,row sep=1em,column sep=1em,minimum width=1em,nodes={text height=1.75ex,text depth=0.25ex}]
              { \unit & B \\
                \fib{r}{b} & E \\
                \unit & B \\};
              \path[-stealth]
                (m-1-1) edge node [above] {} (m-1-2)
                        edge node [left] {} (m-2-1)
                (m-2-1) edge node [above] {} (m-2-2)
                        edge node [left] {} (m-3-1)
                (m-3-1) edge node [above] {$b$} (m-3-2)
                (m-1-2) edge node [right]{$s$} (m-2-2)
                (m-2-2) edge node [right] {$r$} (m-3-2)
                ;
            \end{tikzpicture}\qedhere
            \]
    \end{proof}

\begin{lem}
    The map $R$ has a retraction, given by $F_\sbt : \EM_{\sbt\sbt}(A,n+1) \to \EM_\sbt(A,n+1)$, the map
    that forgets the second point. \qed
\end{lem}

From \cref{proposition:characterizationF}, it follows that $F$ has a retraction.

The following proposition is a strengthening of the main theorem in \citep{shulman}.
Instead of proving a logical equivalence, we provide an equivalence of types.
For this proposition, it is convenient to define the notion of unpointed associated action.
Similarly to \cref{definition:associatedaction}, given a $K(A,n)$-bundle $f : Y \to X$, we define
the \define{unpointed associated action} of $X$ on $A$ as the (unpointed) map $F_\sbt \circ \ceil{f} : X \to \EM_\sbt(A,n+1)$.

\begin{prop}[cf.~\citep{shulman}]
    Let $f : Y \to X$ be a $K(A,n)$-bundle.
    Then the following types are equivalent:
    \begin{enumerate}
        \item The type of homotopies between the unpointed associated action of $X$ on $A$ and the constant map at
            the pointing of $\EM_\sbt(A,n+1)$.
        \item The type of principal fibration structures on $f$.
    \end{enumerate}
\end{prop}
\begin{proof}
    Observe that we have a fiber sequence
    \[
        K(A,n+1)\xrightarrow{i} \EM_{\sbt\sbt}(A,n+1) \xrightarrow{F_{\sbt}} \EM_\sbt(A,n+1),
    \]
    and thus, the first type in the statement is
    equivalent to the type of factorizations of $\ceil{f}$ through $i$.
    Concretely, this is the type $\sm{f' : X \to K(A,n+1)} i \circ f' \sim \ceil{f}$.

    Now, by \cref{remark:classificationanduniqueness} and \cref{proposition:characterizationF}, this last type is equivalent to the type of maps
    $f' : X \to K(A,n+1)$ together with a proof that $i \circ f'$ classifies $f$.
    That is, a pointed map $\psi : Y \to \EM_\sbt(A,n+1)$ together with a homotopy,
    making the following square commute and a pullback:
     \begin{center}
        \begin{tikzpicture}
          \matrix (m) [matrix of math nodes,row sep=2em,column sep=2em,minimum width=2em,nodes={text height=1.75ex,text depth=0.25ex}]
          { Y &  & \EM_\sbt(A,n+1) \\
            X & K(A,n+1) & \EM_{\sbt\sbt}(A,n+1). \\};
          \path[-stealth]
            (m-1-1) edge node [above] {$\psi$} (m-1-3)
                    edge node [left] {$f$} (m-2-1)
            (m-2-1) edge node [above] {$f'$} (m-2-2)
            (m-2-2) edge node [above] {$i$} (m-2-3)
            (m-1-3) edge node [right]{$R$} (m-2-3)
            ;
        \end{tikzpicture}
    \end{center}
    By \cref{lemma:trivialintersection}, and the universal property of pullbacks,
    this is equivalent to having maps $f' : X \to K(A,n+1)$,
    $\psi : Y \to \EM_\sbt(A,n+1)$, and $\phi : Y \to \unit$, homotopies $H$ and $J$,
    filling the left square and the top triangle respectively in the following diagram
     \begin{center}
        \begin{tikzpicture}
          \matrix (m) [matrix of math nodes,row sep=2em,column sep=2em,minimum width=2em,nodes={text height=1.75ex,text depth=0.25ex}]
          { Y & \unit & \EM_\sbt(A,n+1) \\
            X & K(A,n+1) & \EM_{\sbt\sbt}(A,n+1). \\};
          \path[-stealth]
            (m-1-1) edge node [below] {$\phi$} (m-1-2)
                    edge [bend left = 20] node [above] {$\psi$} (m-1-3)
                    edge node [left] {$f$} (m-2-1)
            (m-2-1) edge node [above] {$f'$} (m-2-2)
            (m-1-2) edge node [above] {} (m-1-3)
                    edge node [left] {} (m-2-2)
            (m-2-2) edge node [above] {$i$} (m-2-3)
            (m-1-3) edge node [right]{$R$} (m-2-3)
            ;
        \end{tikzpicture}
    \end{center}
    and a proof that the pasting of the two squares is a pullback.

    The type of maps $Y \to \unit$ is contractible, with center of contraction the canonical map $Y \to \unit$.
    Observe also that the pair $(J, \psi)$ inhabits a contractible type, since $J$ is witnessing the fact that $\psi$ is the composite
    of $\phi$ with the inclusion of the base point $\unit \to \EM_{\sbt}(A,n+1)$.
    So this last type is equivalent to the type of 
    maps $f' : X \to K(A,n+1)$ together with a homotopy $H$ making the left square commute,
    and a proof that the composite square is a pullback.
    Finally, by the two pullback lemma \citep[Lemma~4.1.11]{AKL}, this last type is equivalent to the type of
    proofs that $f$ is pointed principal, as required.
\end{proof}

A similar argument, using pointed maps, pointed homotopies, and \cref{lemma:actionisaction}, proves
an analogous result for pointed $K(A,n)$-bundles between connected types.

\begin{prop}
    \label{proposition:characterizationprincipalfibration}
    Let $f : Y \to X$ be a pointed $K(A,n)$-bundle between connected types.
    Then the following types are equivalent:
    \begin{enumerate}
        \item The type of trivializations of the associated action of $X$ on $A$.
        \item The type of principal fibration structures on $f$.\qed
    \end{enumerate}
\end{prop}


\begin{cor}\label{remark:principalinoneway}
    The type of principal fibration structures on a pointed $K(A,n)$-bundle between
    connected types is a mere proposition.
\end{cor}
\begin{proof}
    This follows from \cref{proposition:characterizationprincipalfibration},
    noting that \cref{lemma:actionisaction} implies that the first of the two
    equivalent types in the statement
    of \cref{proposition:characterizationprincipalfibration} is a mere proposition.
\end{proof}

\subsection{Nilpotent groups}
\label{nilpgroups}

In this section we give a homotopical characterization of nilpotent groups that is familiar from classical homotopy theory.
Before doing this, recall, from \citep[Section~4.5]{BRV},
that short exact sequences of groups of the form $I \to G \to H$
correspond to pointed fiber sequences of Eilenberg--Mac Lane spaces of the form $K(I,1) \to K(G,1) \to K(H,1)$.

\begin{defn}\label{remark:actionofabelia}
Given a short exact sequence of groups $I \to G \to H$, where $I$ is abelian,
we define an action $H \curvearrowright I$ by applying
\cref{definition:associatedaction} together with \cref{lemma:actionisaction}
to the fiber sequence $K(I,1) \to K(G,1) \to K(H,1)$ corresponding to the short exact sequence.
\end{defn}

\begin{lem}\label{lemma:algebraicaction}
    Given a short exact sequence of groups $I \to G \xrightarrow{q} H$ with $I$ abelian,
    the action $G \curvearrowright I$ given in \cref{remark:actionofses} factors as $q$ composed
    with the action $H \curvearrowright I$ given in \cref{remark:actionofabelia}.
\end{lem}
\begin{proof}
    We have equivalences $G \simeq \loopspacesym K(G,1)$, and $H \simeq \loopspacesym K(H,1)$, such that
    multiplication maps to composition of loops. For readability, we will avoid using the equivalences explicitly.
    Assume given a pointed map $f : K(G,1) \to K(H,1)$ that represents $q$.
    Let us recall the setting of \cref{remark:actionofses}.
    Let $I\equiv \fib{\loopspacesym f}{\refl{\ast}}$, and fix $g : G$.
    We then have a map $I \to I$ given by $(l, r) \mapsto (g^{-1} \sq l \sq g, \phi_r)$,
    where $\phi_r$ is the straightforward proof that $g^{-1} \sq l \sq g$ maps to $\refl{\ast}$ under $\loopspacesym f$.
    Our goal is to show that using the action from \cref{remark:actionofabelia}, $\loopspacesym f(g)$ acts as the above map.
    To do this, let us recall the construction from \cref{remark:actionofabelia}.
    We start by considering the composite in \cref{definition:associatedaction}:
    \[
        K(H,1) \to \EM(I,1) \xrightarrow{\sim} \EM_{\sbt\sbt}(I,2) \to \EM_{\sbt}(I,2).
    \]
    Together with $f$, this gives us a map $K(G,1) \to \EM_{\sbt}(I,2)$ that sends
    $\ast$ to $\ttrunc{2}{\suspsym \fib{f}{\ast}}$.
    By looping this map, we get a map $\alpha : G \to \left(\ttrunc{2}{\suspsym \fib{f}{\ast}} =_{\UU_{\sbt}} \ttrunc{2}{\suspsym \fib{f}{\ast}}\right)$.
    The action $G \curvearrowright I$ is obtained by using \cref{remark:equivalenceofEM},
    and recalling that we have
    \[
        \loopspacesym^2 \ttrunc{2}{\suspsym \fib{f}{\ast}} \simeq \loopspacesym \ttrunc{1}{\loopspacesym \suspsym \fib{f}{\ast}}
        \simeq \loopspacesym \fib{f}{\ast} \simeq \fib{\loopspacesym f}{\refl{\ast}} \equiv I
    \]
    This gets us the map $\beta : G \to (I \to I)$ which we want to prove is homotopic to $(l, r) \mapsto (g^{-1} \sq l \sq g, \phi_r)$.

    Now, to prove this, we generalize and use path induction.
    Instead of considering a loop $g : \loopspacesym K(G,1)$, we take a point $x : K(G,1)$ and a path $p : \ast = x$.
    Generalizing the construction from \cref{remark:actionofabelia} (using $\ap{f}$ instead of $\loopspace f$) we get a map
    $(\ast = x) \to (\fib{\ap{f}}{\refl{\ast}} \to \fib{\ap{f}}{\refl{f(x)}})$,
    and it is now enough to show that this map is homotopic to $p \mapsto \big((c,r) \mapsto (p^{-1} \sq c \sq p, \psi_r)\big)$.
    But this is easily proven by path induction, since we only have to verify it in the case where $x \equiv \ast$ and $p \equiv \refl{\ast}$. 
    In that case, on the one hand, the map $(c,r) \mapsto (p^{-1} \sq c \sq p, \psi_r)$ is just the identity on $I$.
    On the other hand, $\alpha$, being the looping of a pointed map, maps $\refl{\ast}$ to $\refl{\ttrunc{2}{\suspsym \fib{f}{\ast}}}$.
    So $\beta$ maps $\refl{\ast}$ to $\idfunc{I}$, since for $K,K' : \EM_{\sbt}(A,2)$, the equivalence $(K =_{\UU_\sbt} K') \simeq (\loopspacesym^2 K =_{\Grp} \loopspacesym^2 K')$
    is given by $\ap{\loopspacesym^2}$.
\end{proof}

\begin{cor}\label{corollary:centralisprincipal}
Using the correspondence between short exact sequences and pointed fiber sequences of Eilenberg--Mac Lane spaces,
a short exact sequence is a \define{central extension} if and only if its corresponding fiber sequence is
a \define{principal fibration}.
\end{cor}
    The statement of the corollary is only a logical equivalence and a priori not an equivalence of types. 
    But notice that being a central extension is a mere proposition,
    and that \cref{remark:principalinoneway} implies that, in this case, being principal is a mere proposition too.
    So the statement is actually an equivalence of types.
\begin{proof}
    Fix a short exact sequence $I \to G \xrightarrow{q} H$. On the one hand, if it corresponds to a principal fibration,
    then $I$ is abelian and the action of \cref{remark:actionofabelia} is trivial,
    so the action of \cref{remark:actionofses} must be too, by \cref{lemma:algebraicaction}.
    On the other hand, since the fact that an element $h : H$ acts trivially as a map $I \to I$ is a mere proposition,
    and since $q$ is surjective, we can assume that $h = q(g)$ for some $g : G$.
    But if we assume that the action of \cref{remark:actionofses} is trivial, then $g$ acts trivially, and, by
    \cref{lemma:algebraicaction}, $h$ must act trivially too.
\end{proof}

\begin{prop}\label{proposition:characterizationnilpgroup}
    Given a group $G$, the following types are equivalent:
    \begin{enumerate}
        \item The type of nilpotent structures on $G$.
        \item The type of factorizations of the map $K(G,1) \to \unit$ as a finite composite of
            pointed principal fibrations involving pointed, connected $1$-types.
    \end{enumerate}
\end{prop}

\begin{proof}
    By the correspondence between central extensions and principal fibrations
    (\cref{corollary:centralisprincipal}) the statement is equivalent to \cref{lemma:injectionssurjections}.
\end{proof}

\subsection{Nilpotent types}
\label{nilpspaces}

This section characterizes nilpotent types in terms of the maps in their Postnikov towers,
as is familiar from classical homotopy theory.

\begin{defn}
    A \define{nilpotent structure} on a pointed, connected type $X$ is given by
    a nilpotent structure for $\pi_1(X)$ and, for each $n>1$, a nilpotent structure for
    the action $\pi_1(X) \curvearrowright \pi_n(X)$.
\end{defn}

Notice that a connected type is merely pointed, and thus it makes sense to ask whether a
connected type merely has a nilpotent structure.

\begin{defn}
    A \define{nilpotent type} is a connected type such that it merely has a nilpotent structure.
\end{defn}

\begin{eg}
    All simply connected types are nilpotent.
\end{eg}

\begin{eg}
    \label{example:truncofnilp}
    The truncation of a nilpotent type is nilpotent.
\end{eg}

We can prove many facts about truncated nilpotent types by inducting over the nilpotency degree.

\begin{defn}
    Given $n\geq 1$ and a pointed, connected, $n$-truncated type $Y$ together with a nilpotent structure, its \define{nilpotency degree}
    is the sum, over $i<n$, of the lengths of the factorizations of the maps $\ttrunc{i+1}{Y} \to \ttrunc{i}{Y}$,
    given by the nilpotent structure.
\end{defn}

Our next goal is to prove the key property of a nilpotent type, namely that all of the maps in its Postnikov
tower factor as finite composites of principal fibrations.
In order to do this, we need to transfer $K(-,n)$-bundles along maps of actions.

Given a map of $\pi_1(X)$-actions $m : \phi \to \psi$, and a $K(A,n)$-bundle over $\phi$,
we can construct a bundle over $\psi$ as follows.

\begin{defn}\label{definition:transferred}
    Let $f : Y \to X$ be a bundle classified by $\overline{f} : X \to \EM(A,n+1)$ and $s_1, s_2 : \prod_{(x:X)}f(x)$.
    Given an action $\psi \equiv (\overline{f}',s_1')$ and a map of actions
    \[
        m : \prd{x:X} (\overline{f}(x),s_1(x)) \to_{\sbt} (\overline{f}'(x),s_1'(x)),
    \]
    the \define{bundle transferred along $m$}, denoted by $m_*(f)$, is the bundle classified by
    $\overline{f}'$, $s_1'$, and $s_2' \defeq \lambda x. m(x)(s_2(x))$.
\end{defn}

Notice that if $f$ is a pointed bundle and $\psi$ is a pointed action, then the transferred bundle is canonically pointed.

\begin{defn}
    \label{remark:functorialityinducedbundle}
    Let $f : Y \to X$ be a (pointed) $K(-,n)$-bundle.
    We define a (pointed) fiberwise map $f \to m_*(f)$.
Under the identification of \cref{remark:descriptionbundles}, this is the map
\[
    \left(\sm{x : X} s_1(x) = s_2(x)\right) \to \left(\sm{x : X} s_1'(x) = s_2'(x)\right)
\]
    given by mapping $(x,p)$ to $(x,r^{-1} \sq \ap{m(x)}(p))$, where $r : m(s_1(x)) = s_1'(x)$
    is the proof that $m$ is pointed.
\end{defn}

The transferred bundle construction has a nice interpretation in terms of cohomology with local coefficients.

\begin{rmk}
Recall from \citep{synthcoh} that $K(A,n)$ represents the \define{$n$-th reduced cohomology group
with coefficients in $A$}. That is, for a pointed type $X$, one defines
\[
    \tilde{H}^{n}(X;A) \defeq \ttrunc{0}{X\to_\sbt K(A,n)}.
\]

More generally, the map $F_\sbt : \EM_{\sbt\sbt}(A,n+1) \to \EM_\sbt(A,n+1)$ represents the \define{$n$-th reduced cohomology group with local coefficients}.
In this setting, a \define{local system} is given by a pointed map $c : X \to_{\sbt} \EM_\sbt(A,n+1)$
(which, by \cref{lemma:actionisaction}, corresponds to an action of $\pi_1(X)$ on $A$, when $X$ is connected), and the
$n$-th reduced cohomology group of $X$ with coefficients in $c$ is defined by
\[
    \tilde{H}^{n}(X;c) \defeq \ttrunc{0}{\text{\LARGE $\Pi$}_{x:X}^\sbt c(x)},
\]
where $\prod^\sbt$ is the pointed Pi-type.

Under the above interpretation, the transferred bundle construction gives a map $m : \tilde{H}^*(X,\phi) \to \tilde{H}^*(X,\psi)$
for every map of actions $m : \phi\to\psi$. This is the functoriality of cohomology with local coefficients 
with respect to the coefficients variable.
\end{rmk}


\begin{defn}\label{remark:fibseqfibers}
Given composable maps $Y \xrightarrow{f} Y' \xrightarrow{g} Y''$
and $y' : Y'$, we define a fiber sequence
\[
    \fib{f}{y'} \to \fib{g\circ f}{g(y')} \to \fib{g}{g(y')},
\]
where we are considering the fiber over $(y', \refl{}) : \fib{g}{g(y')}$.
To construct this fiber sequence,
we take the fibers of all the maps in the bottom right square of the following diagram:
    \begin{center}
        \begin{tikzpicture}
          \matrix (m) [matrix of math nodes,row sep=1.5em,column sep=1em,minimum width=2em,nodes={text height=1.75ex,text depth=0.25ex}]
            {  & \fib{g\circ f}{g(y')} & \fib{g}{g(y')}\\
            \fib{f}{y'} &  Y & Y' \\
            \unit & Y'' & Y'', \\};
          \path[-stealth]
            (m-2-1) edge node [above] {} (m-2-2)
                    edge node [left] {} (m-3-1)
            (m-1-2) edge node [right]{} (m-2-2)
                    edge node [above] {} (m-1-3)
            (m-2-2) edge node [left] {$g\circ f$} (m-3-2)
                    edge node [above] {$f$} (m-2-3)
            (m-3-1) edge node [left] {} (m-3-2)
            (m-3-2) edge [double equal sign distance,-] node [left] {} (m-3-3)
            (m-1-3) edge node [left] {} (m-2-3)
            (m-2-3) edge node [left] {$g$} (m-3-3)
            ;
        \end{tikzpicture}
    \end{center}
and then use the commutativity of limits to get an equivalence between the fiber of
the map $\fib{g\circ f}{g(y')} \to \fib{g}{g(y')}$ and $\fib{f}{y'}$.
\end{defn}

\begin{lem}\label{lemma:compositionKnbundles}
    For every $n\geq 1$, pointed $K(-,n)$-bundles are closed under composition.
\end{lem}
\begin{proof}
    Assume given composable pointed maps $f : Y \to Y'$, $g : Y' \to Y''$ such that $f$ is a pointed $K(A,n)$-bundle and $g$ is a pointed $K(B,n)$-bundle.
    Let $y :Y$, $y' : Y'$, $y'' : Y''$ be the pointings.
    The fiber sequence of \cref{remark:fibseqfibers}, in this case, is equivalent to a fiber sequence of the following form:
    \[
        K(A,n) \to \fib{g\circ f}{y''} \to K(B,n).
    \]
    In particular $\fib{g\circ f}{y''}$ is a pointed, $(n-1)$-connected, $n$-truncated type.
    So it is an Eilenberg--Mac Lane space.
    By looking at the long exact sequence of homotopy groups, we see that
    $g \circ f$ is in fact a $K(C,n)$-bundle, with $C$ an extension of $B$ by $A$.
\end{proof}

\begin{lem}\label{lemma:surjectiveaction}
    Let $X$ be a pointed, connected type, and
    assume given pointed $K(-,n)$-bundles $f : Y \to X$, $f' : Y' \to X$, $g : Y \to Y'$, such that $f = f' \circ g$
    as pointed maps.
    Then the induced map of actions $t_g : (\overline{f}, s_1) \to (\overline{f'},s'_1)$
    is surjective (\cref{definition:mapofactionsdef}).
\end{lem}

\begin{proof}
    Let $x_0 : X$ and $y'_0 : Y'$ be the base points.
    Notice that by the construction of the induced action (\cref{lemma:functorialityaction})
    the map $t_g(x_0) : \overline{f}(x_0) \to \overline{f'}(x_0)$ is equivalent to the second map in the
    following fiber sequence:
\[
    \fib{g}{y'_0} \to \fib{f}{x_0} \to \fib{f'}{x_0}.
\]
    By hypothesis, $\fib{g}{y'_0} \simeq K(A,n)$, $\fib{f}{x_0} \simeq K(C,n)$, and
    $\fib{f'}{x_0} \simeq K(B,n)$.
    This fiber sequence corresponds to a short exact sequence of groups,
    so after looping $n$ times, $t_g(x_0) : \overline{f}(x_0) \to \overline{f'}(x_0)$
    is a surjective map of groups.
\end{proof}

\begin{lem}\label{lemma:ifactionscoincide}
    Assume given commutative triangles of pointed $K(-,n)$-bundles
    \[
        \begin{tikzpicture}
          \matrix (m) [matrix of math nodes,row sep=2em,column sep=2em,minimum width=2em,nodes={text height=1.75ex,text depth=0.25ex}]
            { Y &   & Y' \\
              & X &  \\};
          \path[-stealth]
            (m-1-1) edge node [above] {$g$} (m-1-3)
                    edge node [below] {$f\,\,\,\,$} (m-2-2)
            (m-1-3) edge node [below] {$\,\,\,\,f'$} (m-2-2)
            ;
        \end{tikzpicture}
        \begin{tikzpicture}
          \matrix (m) [matrix of math nodes,row sep=2em,column sep=2em,minimum width=2em,nodes={text height=1.75ex,text depth=0.25ex}]
            { Y &   & Y'' \\
              & X, &  \\};
          \path[-stealth]
            (m-1-1) edge node [above] {$g'$} (m-1-3)
                    edge node [below] {$f\,\,\,\,$} (m-2-2)
            (m-1-3) edge node [below] {$\,\,\,\,f''$} (m-2-2)
            ;
        \end{tikzpicture}
    \]
    and an equivalence of actions $e : (\overline{f'}, s'_1) \simeq (\overline{f''}, s''_1)$,
    such that $e \circ t_g = t_g'$ as maps of actions.
    Then we have a pointed equivalence $h : Y \simeq Y'$ such that
    $h \circ g = g'$ and $f'' \circ h = f'$ as pointed maps.
\end{lem}
\begin{proof}
    Notice that $e(x)(s'_i(x)) = s''_i(x)$, for $i$ either $1$ or $2$.
    This is because on the one hand $e(x) \circ t_g(x) = t_g'(x)$ as maps,
    and on the other hand, $t_g$ and $t_g'$ are doubly pointed, so
    we have equalities $s'_i(x) = t_g(x)(s_i(x))$ and $s''_i(x) = t_g'(x)(s_i(x))$.
    \cref{proposition:functorialitycharacterizationbundles} then implies that $Y$ and
    $Y'$ are equivalent, with the required compatibility.
\end{proof}

\begin{lem}\label{lemma:surjectivemapofbundles}
    Let $X,\,Y,\,Y'$ be pointed, connected types, and $f : Y \to X$, $f' : Y' \to X$ be pointed $K(-,n)$-bundles.
    Given a map $g : Y \to Y'$ such that $f = g \circ f'$, if the action $(\overline{f}, s_1) \to_\act (\overline{f}', s'_1)$
    induced by $g$ is surjective, then $g$ is a $K(-,n)$-bundle, and its associated action is equivalent to $f'$ composed with
    $\ker\left((\overline{f}, s_1) \to_\act (\overline{f}', s'_1)\right)$.
\end{lem}
\begin{proof}
    Using \cref{remark:functorialityinducedbundle}, we replace $g$ by an equivalent map
    \[
     (\idfunc{}, \lambda x.\ap{t_g(x)}) : \left(\sm{x :X} s_1(x) =_{\overline{f}(x)} s_2(x)\right)  \to \left(\sm{x :X} s'_1(x) =_{\overline{f'}(x)} s'_2(x)\right).
    \]
    Now, in general, whenever we have a map $\eta : U \to V$ between types, and points $u_1, u_2 : U$,
    we have an equivalence
    \[
        (u_1 = u_2) \simeq \left(\sm{p : \eta(u_1) = \eta(u_2)} (u_1, \refl{\eta(u_1)}) =_{\fib{\eta}{\eta(u_1)}} \transfib{\fib{\eta}{-}}{p^{-1}}{(u_2, \refl{\eta(u_2)})}\right)
    \]
    such that a path $z$ gets mapped to a pair with $\ap{\eta}(z)$ as its first coordinate.
    This implies that the fiber of $g$ is equivalent to
    \[
        (s_1(x), \refl{s'_1(x)}) =_{\fib{t_g(x)}{s'_1(x)}} \transfib{\fib{t_g(x)}{-}}{p^{-1}}{(s_2(x), \refl{s_2(x)})}.
    \]
    Since there is a fiber sequence $\fib{t_g(x)}{s'_1(x)} \to \overline{f}(x) \xrightarrow{t_g} \overline{f'}(x)$,
    if the action induced by $g$ is surjective, $\fib{t_g(x)}{s'_1(x)}$ must be $(n+1)$-truncated and $n$-connected.
    This means that $g$ is a $K(-,n)$-bundle, since its fibers are path spaces of an $(n+1)$-dimensional
    Eilenberg--Mac Lane space. This proves the first statement.

    By the very description of $(\idfunc{}, \lambda x.\ap{t_g(x)})$ as a $K(-,n)$-bundle, we see that its action is given by
    mapping $(x,p)$ to the pointed type $(\fib{t_g(x)}{s'_1(x)}, (s_1(x), \refl{s'_1(x)}))$,
    which is precisely $\ker\left((\overline{f}, s_1) \to_\act (\overline{f}', s'_1)\right)$ applied to $x$.
    To finish the proof, notice that under the equivalence $Y' \simeq\left(\sm{x :X} s'_1(x) =_{\overline{f'}(x)} s'_2(x)\right)$,
    $f'$ corresponds to the projection $(x,p) \mapsto p$.
\end{proof}

\begin{prop}\label{proposition:structuretofactorization}
    Let $f : Y \to X$ be a $K(A,n)$-bundle between pointed, connected types living over an action $\phi : X \to_\sbt \EM_\sbt(A,n+1)$.
    Given a nilpotent structure on $\phi$ there is a factorization of $f$ as a finite composite of principal fibrations.
\end{prop}
\begin{proof}
    Let $\unit$ denote the trivial $\pi_1(X)$-action on the trivial group.
    Given a filtration of normal subactions
    \[
        \unit \equiv \phi_0 \lhd \cdots \lhd \phi_k \equiv \phi
    \]
    all of them normal in $\phi$, we can construct a sequence of maps of actions
    \[
        \phi \simeq \phi'_0 \to \cdots \to \phi'_k \simeq \unit
    \]
    that induce surjective maps after looping $n$ times.
    Analogously to \cref{proposition:characterizationnilpgroup}, we get this by defining $\phi'_i \defeq \phi/\phi_i$.
    We can then apply \cref{definition:transferred} to the quotient map $\phi \to \phi'_{k-1}$, to get a factorization
    \[
        \begin{tikzpicture}
          \matrix (m) [matrix of math nodes,row sep=2em,column sep=2em,minimum width=2em,nodes={text height=1.75ex,text depth=0.25ex}]
          { Y &   & Y_{k-1} \\
              & X &  \\};
          \path[-stealth]
            (m-1-1) edge node [above] {$f^{(1)}$} (m-1-3)
                    edge node [below] {$f\,\,\,\,$} (m-2-2)
            (m-1-3) edge node [below] {$\,\,\,\,f_{k-1}$} (m-2-2)
            ;
        \end{tikzpicture}
    \]
    such that the associated action of $f_{k-1}$ is $\phi'_{k-1}$.
    Moreover, by \cref{lemma:surjectivemapofbundles}, $f^{(1)}$ is a $K(-,n)$-bundle, and its associated action is $\phi_{k-1} \circ f_{k-1}$.

    By hypothesis, the action associated to $f_{k-1}$ is trivial, and thus $f_{k-1}$ is a principal fibration by
    \cref{proposition:characterizationprincipalfibration}.
    Notice that composing each $\phi_i$ with $f_{k-1}$, we get a filtration of the action $\phi_{k-1} \circ f_{k-1}$
    \[
        \unit \equiv \phi_0 \circ f_k \lhd \cdots \lhd \phi_{k-1} \circ f_{k-1}.
    \]
    We can now proceed inductively, by applying the same argument to $f^{(1)}$.
    In the end, we get a factorization of $f$ as a sequence of principal fibrations
    \[
        Y \to Y_1 \to \cdots \to Y_{k-1} \to X,
    \]
    as needed.
\end{proof}

\begin{prop}\label{proposition:factorizationtostructure}
    Let $f : Y \to X$ be a $K(A,n)$-bundle between pointed, connected types living over an action $\phi : X \to_\sbt \EM_\sbt(A,n+1)$.
    Given a factorization of $f$ as a finite composite of principal fibrations, there is a nilpotent structure on $\phi$.
\end{prop}
\begin{proof}
    Suppose we are given a factorization of $f$
    \[
        Y \equiv Y_0 \xrightarrow{g_1} Y_1 \xrightarrow{g_2} \cdots \xrightarrow{g_{k-1}} Y_{k-1} \xrightarrow{g_k} Y_k \equiv X.
    \]
    Let $f_i : Y_i \to X$ denote the composite of the maps $g_{i+1}$ to $g_k$, so that we have commuting triangles
    \[
        \begin{tikzpicture}
          \matrix (m) [matrix of math nodes,row sep=2em,column sep=2em,minimum width=2em,nodes={text height=1.75ex,text depth=0.25ex}]
            { Y_i &   & Y_{i+1} \\
              & X, &  \\};
          \path[-stealth]
            (m-1-1) edge node [above] {$g_{i+1}$} (m-1-3)
                    edge node [below] {$f_i\,\,\,\,$} (m-2-2)
            (m-1-3) edge node [below] {$\,\,\,\,f_{i+1}$} (m-2-2)
            ;
        \end{tikzpicture}
    \]
    where $f_k : Y_k \to X$ is the identity.
    By \cref{lemma:compositionKnbundles}, the map $f_i$ is a $K(-,n)$-bundle for every $i$.

    For each $i$, consider the actions associated to the maps $f_i$ and $f_{i+1}$, and call them $\phi'_i$ and $\phi'_{i+1}$.
    From \cref{lemma:functorialityaction} we know that the map $g_{i+1}$ induces a map of actions $\phi'_i \to_\act \phi'_{i+1}$,
    and, from \cref{lemma:surjectiveaction}, we know this map is surjective.
    By composing these maps, we obtain, for each $i$, a surjective map $\phi'_0 \to \phi'_i$.
    Let $\phi_i$ be the kernel of this map. This gives us a sequence of normal subactions
    \[
        \phi_0 \lhd \phi_1 \lhd \cdots \lhd \phi_k
    \]
    such that $\phi_k$ is equivalent to $\phi'_0$, the action associated to $f$.
    So this sequence in fact gives us a filtration of the action associated to $f$.
    Now, analogously to \cref{proposition:characterizationnilpgroup}, the quotient
    $\phi_{i+1}/\phi_i$ is equivalent to $k_i$, the kernel of the map $\phi'_i \to_\act \phi'_{i+1}$.
    So it remains to show that $k_i$ is a trivial action.

    The action associated to $g_i$ is surjective. \cref{lemma:surjectivemapofbundles}
    then implies that the action associated to $g_i$ is equivalent to $k_i \circ f_i$.
    Moreover, $g_i$ is principal, so its associated action is trivial,
    by \cref{proposition:characterizationprincipalfibration}.
    Since $f_i$ is homotopy surjective, it follows that $k_i$ is a trivial action, as needed.
\end{proof}

\begin{thm}\label{theorem:mainchar}
    Let $X$ and $Y$ be pointed, connected types, and $f : X \to Y$ a pointed $K(-,n)$-bundle.
    Then the type of nilpotent structures on $(\overline{f},s_1)$ is equivalent to the type
    of factorizations of $f$ as a finite composite of principal fibrations.
\end{thm}
\begin{proof}
    We constructed maps back and forth in \cref{proposition:factorizationtostructure} and \cref{proposition:structuretofactorization}.
    The fact that given a nilpotent structure on $(\overline{f},s_1)$ we obtain the same
    structure after applying \cref{proposition:structuretofactorization} and then \cref{proposition:factorizationtostructure}
    is clear by construction.

    The other composite follows from \cref{lemma:ifactionscoincide}.
\end{proof}

As a direct corollary we get the main characterization of nilpotent types.

\begin{thm}\label{theorem:nilpotentchar}
    For a connected type $X$ the following are equivalent:
    \begin{itemize}
        \item The type $X$ is nilpotent.
        \item Each map $\ttrunc{n+1}{X} \to \ttrunc{n}{X}$ in the Postnikov tower of $X$
            merely factors as a finite composite of principal fibrations.
    \qed
    \end{itemize}
\end{thm}

\section{Cohomology isomorphisms}
\label{section:cohiso}

In this section we prove that if a pointed map between nilpotent types induces
a cohomology isomorphism with coefficients in every abelian group, then it induces an isomorphism in all homotopy groups.
We use this to prove that the suspension of an $\infty$-connected map between connected types is $\infty$-connected.
The main ideas in this section appear in \citep{shulman}.

\begin{defn}
    A pointed map $f : Y \to_\sbt X$ between pointed, connected types is a \define{cohomology isomorphism}
    if for every abelian group $A$ we have that $f^* : \tilde{H}^\bullet(X;A) \to \tilde{H}^\bullet(Y;A)$
    is an isomorphism.
\end{defn}

Notice that we are using reduced cohomology. The main result of this section can also be stated and proved
using non-reduced cohomology.

\begin{defn}
    A map $f : Y \to X$ is \define{$\infty$-connected} if for all $n : \N$ we have that
    $\ttrunc{n}{f} : \ttrunc{n}{Y} \to \ttrunc{n}{X}$ is an equivalence.
\end{defn}

It is straightforward to see that a pointed map between pointed, connected types is $\infty$-connected
if and only if it induces an isomorphism in all homotopy groups.

\begin{lem}
    \label{remark:inftyconnimplies}
    A map $f : X \to_\sbt Y$ between pointed, connected types is a
    cohomology isomorphism if and only if the map $\ttrunc{n}{f} : \ttrunc{n}{X} \to_\sbt \ttrunc{n}{Y}$ is a cohomology isomorphism
    for all $n : \N$.
    In particular, any $\infty$-connected, pointed map between pointed, conected types is a cohomology isomorphism.
\end{lem}
\begin{proof}
    The second statement follows at once from the first one.
    The first one follows from the fact that
    cohomology isomorphisms are defined by mapping into truncated types.
\end{proof}

\begin{lem}\label{lemma:eilenbergmaclanearelocal}
    For any cohomology isomorphism $f : Y \to_\sbt X$, abelian group $A$, and $n : \N$,
    the precomposition map $f^* : (X \to_\sbt K(A,n)) \to (Y \to_\sbt K(A,n))$ is an equivalence.
\end{lem}
\begin{proof}
    Since $K(A,n)$ is $n$-truncated, the domain and codomain of $f^*$ are $n$-truncated, so we
    can use the truncated Whitehead theorem \citep[Theorem~8.8.3]{hottbook}.
    It is thus enough to check that $f^*$ induces a bijection after $0$-truncating, and an
    isomorphism in $\pi_k$ for every pointing $g : X \to_\sbt K(A,n)$ and every $k : \N$.

    The statement about $\pi_0$ is clear, since after $0$-truncating, we get
    the induced map on cohomology $\tilde{H}^n(X;A) \to \tilde{H}^n(Y;A)$, which is an isomorphism by hypothesis.

    Given $g : X \to_\sbt K(A,n)$, notice that
    \[
        \loopspacesym(X \to_\sbt K(A,n),g) \simeq \loopspacesym(X \to_\sbt K(A,n), (\lambda x. \ast))
    \]
    since $X \to_\sbt K(A,n)$ is a group object, so any loop at $g$ can be multiplied pointwise by
    $g^{-1}$ to get a loop at the constant map $(\lambda x. \ast)$.
    The same can be argued about $Y \to_\sbt K(A,n)$, using $f^*(g)$ instead of $g$.
    Under this correspondence, given $k>0$, $\ttrunc{0}{\loopspacesym^k f^*}$ corresponds to the induced map
    in cohomology $\tilde{H}^{n-k}(X;A) \to \tilde{H}^{n-k}(Y;A)$, which again is an isomorphism by hypothesis.
\end{proof}

\begin{lem}\label{lemma:nilpotentarelocal}
    For any cohomology isomorphism $f : Y \to_\sbt X$ and truncated, pointed, nilpotent type $Z$, 
    the precomposition map $f^* : (X \to_\sbt Z) \to (Y \to_\sbt Z)$ is an equivalence.
\end{lem}
\begin{proof}
    Since we have to prove a mere proposition, we can assume that we have a nilpotent structure for $Z$.
    We can then prove the result by induction on the nilpotency degree of $Z$.
    If $Z$ is contractible, the statement is clear. So let us assume that we have a principal fibration
    $K(A,n) \to Z \to Z' \to K(A,n+1)$.
    By mapping out of $f : Y \to X$, we obtain a map of fiber sequences
    \[
        \begin{tikzpicture}
          \matrix (m) [matrix of math nodes,row sep=2em,column sep=3em,minimum width=2em,nodes={text height=1.75ex,text depth=0.25ex}]
            { (X \to_\sbt Z) & (X \to_\sbt Z') & (X \to_\sbt K(A,n+1)) \\
              (Y \to_\sbt Z) & (Y \to_\sbt Z') & (Y \to_\sbt K(A,n+1)) \\};
          \path[->]
            (m-1-1) edge [right hook->] node [right] {} (m-1-2)
                    edge node [right] {} (m-2-1)
            (m-1-2) edge node [right] {} (m-1-3)
                    edge node [right] {} (m-2-2)
            (m-2-2) edge node [right] {} (m-2-3)
            (m-2-1) edge [right hook->] node [right] {} (m-2-2)
            (m-1-3) edge node [right] {} (m-2-3)
            ;
        \end{tikzpicture}
    \]
    By \cref{lemma:eilenbergmaclanearelocal}, the right vertical map is an equivalence, and by
    inductive hypothesis, the middle vertical map is an equivalence too.
    So we conclude that the left vertical map is an equivalence, as needed.
\end{proof}

\begin{thm}\label{theorem:cohomologyiso}
    A map between pointed, nilpotent types is a cohomology isomorphism if and only if it is $\infty$-connected.
\end{thm}
\begin{proof}
    Any $\infty$-connected map is a cohomology isomorphism, by \cref{remark:inftyconnimplies}.

    For the other implication, notice that, for all $n : \N$, the map $\ttrunc{n}{f} : \ttrunc{n}{X} \to \ttrunc{n}{Y}$
    is a cohomology isomorphism between pointed, nilpotent, $n$-truncated types. This follows from
    \cref{example:truncofnilp} and \cref{remark:inftyconnimplies}.
    Using \cref{lemma:nilpotentarelocal} and a Yoneda-type argument, we deduce that $\ttrunc{n}{f}$ is an equivalence
    for every $n : \N$, as required.
\end{proof}

As an interesting corollary, we get the following.

\begin{thm}
    The suspension of an $\infty$-connected map between connected types is $\infty$-connected.
\end{thm}
\begin{proof}
    Suppose given an $\infty$-connected map $f : X \to Y$ between connected types.
    Since we are proving a mere proposition, we can assume that $X$, $Y$ and $f$ are pointed.
    From \cref{remark:inftyconnimplies}, it follows that $f$ is a cohomology isomorphism.
    Since cohomology isomorphisms are defined by mapping into Eilenberg--Mac Lane spaces, and these
    are closed under taking loop spaces, the suspension of a cohomology isomorphism is a cohomology
    isomorphism, by the adjunction between suspension and loop space.
    So $\suspsym f : \suspsym X \to \suspsym Y$ is a cohomology isomorphism between simply connected types.
    Since any simply connected type is nilpotent, it follows that $\suspsym f$ is $\infty$-connected,
    as needed.
\end{proof}

\section{Localization}
\label{section:localization}

This section is about the localization of nilpotent types away from sets of numbers.
Let $S : \N \to \N$ be a function enumerating a set of natural numbers.
A group $G$ is \define{$\degg(S)$-local} if, for each $n : \N$, the power map $S(n) : G \to G$ is a bijection.
A morphism of groups $G \to_\Grp G'$ is an \define{algebraic localization} of $G$ away from $S$
if it is the initial morphism into a $\degg(S)$-local group.
For a natural number $m$ and a pointed type $X$, let $m : \loopspacesym X \to \loopspacesym X$
denote the map that sends a loop $l$ to $l^m$.
In \citep{CORS}, the reflective subuniverse of $\degg(S)$-local types is defined.
A type $X$ is \define{$\degg(S)$-local} if, for each $n : \N$, the map $S(n) : \loopspacesym X \to \loopspacesym X$
is an equivalence. In Homotopy Type Theory, this localization plays the role of localization away from the set $S$
of classical homotopy theory.

The main goal of the section is to prove that, when applied to a nilpotent type, $\degg(S)$-localization
localizes the homotopy groups of the type.

In \cref{reflsub} we recall the theory of reflective subuniverses and separated types.
In \cref{lochomgroups} we describe the effect of localization on the homotopy groups of a nilpotent type.

\subsection{Reflective subuniverses and separated types}
\label{reflsub}

We start with a few general results about reflective subuniverses. For details we refer the reader to \citep{RSS} and \citep{CORS}.
We will need the notion of separated type for a reflective subuniverse $L$.

\begin{defn}[{\citep[Definition~2.13]{CORS}}]
    Let $L$ be a reflective subuniverse and let $X : \UU$ be a type.
    We say that $X$ is \define{$L$-separated} if its identity types are $L$-local types.
    We write $L'$ for the subuniverse of $L$-separated types.\qed
\end{defn}

Separated types form a reflective subuniverse.

\begin{thm}[{\citep[Theorem~2.26]{CORS}}]
    For any reflective subuniverse $L$, the subuniverse of $L$-separated types is again reflective.\qed
\end{thm}

Given a reflective subuniverse $L$ and a type $X : \UU$, we denote the unit of the $L'$-localization of
$X$ by $\eta' : X \to L' X$.

Given any family of maps $f : \prod_{i : I} A_i \to B_i$, there is an associated reflective subuniverse $L_f$
of $f$-local types \citep[Theorem~2.16]{RSS}. As proven in \citep[Lemma~2.15]{CORS}, the separated types with respect to
the subuniverse of $f$-local types are precisely the subuniverse of $\suspsym f$-local types. Here $\susp f$ denotes
the family $\prod_{i : I} \suspsym A_i \to \suspsym B_i$ given by suspending $f_i$ for each $i : I$.
Moreover, in the case of localization away from sets of numbers, we have the following:
\begin{prop}[{see \citep[Theorem~4.11]{CORS}}]
    \label{proposition:L'simplyconnected}
        If $X$ is simply connected, then the natural map $L_{\suspsym \degg(S)}X \to L_{\degg(S)} X$ is an equivalence.
\end{prop}

\begin{defn}
    Given a reflective subuniverse $L$, we say that a map $f : X \to Y$ is an \define{$L$-equivalence} if
    $L f : L X \to L Y$ is an equivalence.
\end{defn}

\begin{defn}
    We say that a reflective subuniverse $L$ \define{preserves a fiber sequence} $F \to E \xrightarrow{g} B$ if the induced map
    $F \to \fib{L g}{\eta b}$ is an $L$-localization. Here $b : B$ is the base point of $B$.
    In this case we will also say that $L$ \define{preserves the fiber} of $g$.
\end{defn}

We now prove some useful exactness properties of $L'$-localization.

\begin{lem}\label{lemma:Lequivalencetotalspaces}
Let $P:L'X\to \UU$ be a type family over $L'X$. 
Then the map
\begin{equation*}
f:\Big(\sm{x:X} P(\eta'(x))\Big)\to \Big(\sm{y:L'X}P(y)\Big)
\end{equation*}
given by $(x,p)\mapsto (\eta'(x),p)$ is an $L$-equivalence. 
\end{lem}

\begin{proof}
Using the characterization of $L$-equivalences of \citep[Lemma~2.9]{CORS}, it is enough to prove that $f$
induces an equivalence of mapping spaces whenever we map into an $L$-local type.
Assume given an $L$-local type $Z$ and notice that we have the following factorization of $\precomp{f}$:
    \begin{align*}
        \left(\sm{y : L'X} P(y)\right) \to Z &\simeq \prd{y : L'X} P(y) \to Z\\
                                         &\simeq \prd{x : X} P(\eta' x) \to Z \\
                                         &\simeq \left(\sm{x:X} P(\eta' x)\right) \to Z ,
    \end{align*}
where, in the second equivalence, we use \citep[Proposition~2.22]{CORS} together with the fact that
$P(y) \to Z$ is $L$-local, since $Z$ is.
\end{proof}

\begin{prop}\label{proposition:preservationfibersequences}
    Given a fiber sequence $F \to E \xrightarrow{f} X$, there is a map of fiber sequences
    \[
        \begin{tikzpicture}
          \matrix (m) [matrix of math nodes,row sep=2em,column sep=3em,minimum width=2em,nodes={text height=1.75ex,text depth=0.25ex}]
            { F & E & X \\
              LF & E' & L'X \\};
          \path[->]
            (m-1-1) edge [right hook->] node [right] {} (m-1-2)
                    edge node [left] {$\eta$} (m-2-1)
            (m-1-2) edge node [right] {} (m-1-3)
                    edge node [right] {} (m-2-2)
            (m-2-2) edge node [right] {} (m-2-3)
            (m-2-1) edge [right hook->] node [right] {} (m-2-2)
            (m-1-3) edge node [left] {$\eta'$} (m-2-3)
            ;
        \end{tikzpicture}
    \]
    such that the type $E'$ is $L'$-local and the middle vertical map is an $L$-equivalence.
    Here $F$ is the fiber over the base point $x_0 : X$, and $LF$ is the fiber over $\eta'(x_0)$.
\end{prop}

\begin{proof}
For the proof it is more convenient to work with type families.
We start by constructing the map of fiber sequences.
Assume given a type $X$ and a type family $P : X \to \UU$, with total space $E \defeq \sm{x : X}P(x)$.
By \citep[Lemma~2.19]{CORS}, the composite $L \circ P : X \to \UU_L$ can be extended along the localization $\eta' : X \to L'X$ as follows:
\[
    \begin{tikzpicture}
      \matrix (m) [matrix of math nodes,row sep=2em,column sep=3em,minimum width=2em,nodes={text height=1.75ex,text depth=0.25ex}]
      { X & \UU \\
        L' X & \UU_L . \\};
      \path[->]
        (m-1-1) edge node [above] {$P$} (m-1-2)
                edge node [left] {$\eta'$} (m-2-1)
        (m-2-1) edge [dashed] node [above] {$P'$} (m-2-2)
        (m-1-2) edge node [right]{$L$} (m-2-2)
        ;
    \end{tikzpicture}
\]
In particular, we have $t : \prd{x : X} L(P(x)) \simeq P'(\eta' x)$ and
the localization units induce a map
\begin{align*}
    \sm{x : X} P(x) &\llra{l} \sm{y : L'X} P'(y)\\
             (x,p)  &\longmapsto (\eta' x, t(\eta p)).
\end{align*}
Define $E' \defeq \sm{y : L'X} P'(y)$. This gives us the map of fiber sequences.

Now, $E'$ is $L'$-local, by \citep[Lemma~2.21]{CORS}, so to conclude the proof
we must show that $l : E \to E'$ is an $L$-equivalence.
To see this observe that the $L$-localization of $l$ factors as the following chain of equivalences
\begin{align*}
    L\left(\sm{x : X} P(x)\right) &\simeq L\left(\sm{x : X} L(P(x))\right)\\
                       &\simeq L\left(\sm{x : X} P'(\eta' x)\right)\\
                       &\simeq L\left(\sm{y : L'X} P'(y)\right)
\end{align*}
using~\citep[Theorem~1.24]{RSS} in the first equivalence, the map $t$ in the second equivalence
and \cref{lemma:Lequivalencetotalspaces} in the third one.
\end{proof}

\begin{lem}\label{lemma:localizationpreserveslocal}
For $S : \N \to \N$ and $L$ any reflective subuniverse,
the $L'$-localization of a $\degg(S)$-local type is $\degg(S)$-local.
\end{lem}

\begin{proof}
Let $X$ be a $\degg(S)$-local type.
Fix $n : \N$ and let $k \equiv S(n)$.
We must show that $k : \loopspacesym L'X \to \loopspacesym L'X$
is an equivalence for each base point $x' : L'X$.
Since being an equivalence is a mere proposition, and the localization unit $X \to L'X$ is surjective \citep[Lemma~2.17]{CORS},
we can assume that $x' = \eta(x)$ for some $x : X$.
Consider the square
\[
    \begin{tikzpicture}
      \matrix (m) [matrix of math nodes,row sep=2.5em,column sep=3em,minimum width=2em,nodes={text height=1.75ex,text depth=0.25ex}]
      { L{\loopspacesym X} & \loopspacesym L'X \\
        L{\loopspacesym X} & \loopspacesym L'X . \\};
      \path[->]
        (m-1-1) edge node [left] {$Lk$} (m-2-1)
        (m-2-1) edge node [above] {$\sim$} (m-2-2)
        (m-1-2) edge node [right]{$k$} (m-2-2)
        (m-1-1) edge node [above] {$\sim$} (m-1-2)
        ;
    \end{tikzpicture}
\]
To show that the square commutes, it is enough to check it after precomposing
with the unit $\loopspacesym X \to L \loopspacesym X$, and this is clear.
Finally, the map on the left is an equivalence by hypothesis, and thus $L'X$ is $\degg(S)$-local.
\end{proof}

\subsection{Localization of homotopy groups}
\label{lochomgroups}

In this section we prove that $\degg(S)$-localization is well behaved when applied to nilpotent types.
In particular, we prove that $\degg(S)$-localization localizes the homotopy groups of a nilpotent type algebraically.

\begin{lem}\label{lemma:sigmaclosed}
    The total space of a fibration with $\degg(S)$-local fibers
    and simply connected $\degg(S)$-local base is $\degg(S)$-local.
\end{lem}
\begin{proof}
    Let $P : B \to \UU_{\degg(S)}$ be a fibration over a simply connected, $\degg(S)$-local type.
    We will prove that for every $x : \sm{b : B}P(b)$ and every $k : \N$,
    the map $S(k) : \loopspacesym\left(\sm{b : B}P(b), x\right) \to \loopspacesym\left(\sm{b : B}P(b), x\right)$
    is an equivalence.
    The point $x : \sm{b : B}P(b)$ corresponds to a point $b : B$ together with a point $f : P(b)$ in its fiber,
    and thus induces a fiber sequence
    \[
        \loopspacesym(P(b),f) \to \loopspacesym\left(\sm{b : B}P(b), x\right) \to \loopspacesym(B,b),
    \]
    where, as usual, the fiber is taken to be the fiber over $\refl{b} : \loopspacesym(B,b)$.
    By naturality, this fiber sequence maps to itself by the map $S(k)$.
    By hypothesis, this map of fiber sequences is an equivalence for the base, and thus it is an equivalence for the total
    space if and only if it is a fiberwise equivalence.
    Since the base is connected by hypothesis, it is enough to check that $S(k)$ is an equivalence for
    the fiber over $\refl{b} : \loopspacesym(B,b)$. And this last fact holds by hypothesis.
\end{proof}

\begin{prop}\label{proposition:preservesprincipalfibrations}
    Let $n\geq 1$ and let $Y \to X$ be a pointed principal $K(A,n)$-bundle.
    In particular, we have a classifying map $X \to K(A,n+1)$, and fiber sequences
    $K(A,n) \to Y \to X$ and $Y \to X \to K(A,n+1)$.
    Then, $\degg(S)$-localization preserves both fiber sequences.
    Moreover, the localized map $L_{\degg(S)} Y \to L_{\degg(S)} X$ is a pointed principal $K(A',n)$-bundle,
    where $A'$ is the algebraic localization of $A$.
\end{prop}
\begin{proof}
    Notice that it is enough to show that $\degg(S)$-localization preserves the fiber sequence $Y \to X \to K(A,n+1)$.
    This is because, if $A'$ is the algebraic localization of
    $A$ away from $S$, looping the localization map $K(A,n+1) \to K(A',n+1)$ 
    gives us the localization map $K(A,n) \to K(A',n)$ by \citep[Corollary~4.13]{CORS} and \citep[Theorem~4.22]{CORS}.
    This implies that $L_{\degg(S)} Y \to L_{\degg(S)} X$ is classified by a map $L_{\degg(S)} X \to K(A', n+1)$,
    and that $\degg(S)$-localization preserves the fibration $K(A,n) \to Y \to X$.

    Consider the fiber sequence $Y \to X \to K(A,n+1)$.
    Using \citep[Lemma~2.15]{CORS} and applying \cref{proposition:preservationfibersequences} to this fiber sequence we get a map
    of fiber sequences
    \begin{center}
        \begin{tikzpicture}
            \matrix (m) [matrix of math nodes,row sep=2em,column sep=2em,minimum width=2em,nodes={text height=1.75ex,text depth=0.25ex}]
          { Y & X & K(A,n+1) \\
            L_{\degg(S)} Y & X' & L_{\susp \degg(S)} K(A,n+1), \\};
          \path[-stealth]
            (m-1-1) edge node [above] {} (m-1-2)
                    edge node [left] {} (m-2-1)
            (m-2-1) edge node [above] {} (m-2-2)
            (m-1-2) edge node [right]{} (m-1-3)
                    edge node [left] {} (m-2-2)
            (m-2-2) edge node [right]{} (m-2-3)
            (m-1-3) edge node [above] {} (m-2-3)
            ;
        \end{tikzpicture}
    \end{center}
    where the first two vertical maps are $\degg(S)$-equivalences.
    Since $K(A,n+1)$ is simply connected, the third vertical map is actually a $\degg(S)$-localization by \citep[Theorem~4.11]{CORS},
    so it is enough to show that $X'$ is $\degg(S)$-local.
    To see this, use again the fact that $K(A,n+1)$ is simply connected, and that
    $\degg(S)$-localization preserves simply connectedness \citep[Corollary~4.14]{CORS}, to deduce that the bottom fiber sequence
    satisfies the hypothesis of \cref{lemma:sigmaclosed}.
    We thus see that $L_{\degg(S)}Y \to L_{\degg(S)}X$ is the fiber of 
    the localization of the classifying map $X \to K(A,n+1)$, as needed.
\end{proof}

    Notice that the above proof can be used to prove a slight strengthening of \citep[Theorem~4.16]{CORS}, namely, that
    $\degg(S)$-localization preserves fiber sequences with simply connected base.

    As an application of \cref{proposition:preservesprincipalfibrations}, we give a concrete
    construction of the $\degg(S)$-localization of a pointed,
    truncated type with a nilpotent structure.

\begin{defn}\label{remark:localizationtruncated}
    We give an inductive construction of the $\degg(S)$-localization of
    a pointed, truncated type with a nilpotent structure.
    For the inductive step, this construction uses \cref{proposition:preservesprincipalfibrations}
    and the $\degg(S)$-localization of Eilenberg--Mac Lane spaces of abelian groups, which is done in \citep[Section~4.5]{CORS}.
    This construction is analogous to the classical one (e.g.~\citep{MayPonto}).

    Given a pointed, truncated type $Y$ with a nilpotent structure,
    we induct over the nilpotency degree of $Y$.
    If $Y$ is a point, then it is its own $\degg(S)$-localization.
    If $Y$ is the fiber of a map $Y' \to K(A,n+1)$, then we localize the map, by applying the inductive
    hypothesis to $Y'$ and the localization of Eilenberg--Mac Lane spaces to $K(A,n+1)$.
    By \cref{proposition:preservesprincipalfibrations}, $\degg(S)$-localization preserves the fiber of the
    localized map, so the localization of $Y$ is just the fiber of the localized map.
\end{defn}

\begin{cor}\label{corollary:localizationpreservestruncated}
    For $n \geq 2$, $\degg(S)$-localization of a pointed, $n$-truncated, nilpotent type is $n$-truncated.
\end{cor}
\begin{proof}
    Since being $n$-truncated is a mere proposition, we can assume that we have a nilpotent structure for the nilpotent type.
    The construction of \cref{remark:localizationtruncated} finishes the proof.
\end{proof}

Although we haven't yet proven that the $\degg(S)$-localization of a nilpotent type is nilpotent,
we can already prove the following fact.

\begin{cor}\label{corollary:commutativitylocalizationtruncation}
    For $n\geq 2$, $n$-truncation and $\degg(S)$-localization commute, when restricted to pointed, nilpotent types.
\end{cor}
\begin{proof}
    From \citep[Lemma~4.17]{CORS} we know that the $n$-truncation of a $\degg(S)$-local type is $\degg(S)$-local.
    It follows from \citep[Lemma~2.11]{CORS} that for any type $X$ we have $\ttrunc{n}{L_{\degg(S)}X}\simeq\ttrunc{n}{L_{\degg(S)}\ttrunc{n}{X}}$.
    The result follows from observing that \cref{corollary:localizationpreservestruncated} implies that,
    if $X$ is nilpotent, we have $\ttrunc{n}{L_{\degg(S)}\ttrunc{n}{X}} \simeq L_{\degg(S)}\ttrunc{n}{X}$.
\end{proof}

Our goal now is to show that $\degg(S)$-localization preserves composites of principal fibrations.
To this end, we prove a slightly more general result.

\begin{lem}\label{lemma:inducedfibseqisprincipal}
    Given pointed, composable maps $Y \xrightarrow{f} Y' \xrightarrow{g} Y''$, if $f$ is a pointed principal $K(A,n)$-bundle,
    then the map $\fib{f}{y'} \to \fib{g\circ f}{g(y')}$ of \cref{remark:fibseqfibers} is also a pointed principal $K(A,n)$-bundle.
\end{lem}
\begin{proof}
    Let $f' : Y' \to K(A,n+1)$ classify $f$. 
    By taking fibers of all the maps in the square
    \begin{center}
        \begin{tikzpicture}
          \matrix (m) [matrix of math nodes,row sep=1.5em,column sep=1.5em,minimum width=2em,nodes={text height=1.75ex,text depth=0.25ex}]
          { Y' & K(A,n+1) \\
            Y'' & \unit \\};
          \path[-stealth]
            (m-1-1) edge node [above] {$f'$} (m-1-2)
                    edge node [left] {$g$} (m-2-1)
            (m-2-1) edge node [above] {} (m-2-2)
            (m-1-2) edge node [right]{} (m-2-2)
            ;
        \end{tikzpicture}
    \end{center}
    we obtain a diagram equivalent to the following
    \begin{center}
        \begin{tikzpicture}
          \matrix (m) [matrix of math nodes,row sep=1.5em,column sep=1.5em,minimum width=2em]
          { \fib{g\circ f}{g(y')} & \fib{g}{g(y')} & K(A,n+1)\\
            Y & Y' & K(A,n+1) \\
            Y'' & Y'' & \unit. \\};
          \path[-stealth]
            (m-1-1) edge node [above] {} (m-1-2)
                    edge node [left] {} (m-2-1)
            (m-2-1) edge node [above] {$f$} (m-2-2)
                    edge node [left] {$g\circ f$} (m-3-1)
            (m-1-2) edge node [right]{} (m-2-2)
                    edge node [above] {} (m-1-3)
            (m-2-2) edge node [left] {$g$} (m-3-2)
                    edge node [above] {$f'$} (m-2-3)
            (m-3-1) edge [double equal sign distance,-] node [left] {} (m-3-2)
            (m-3-2) edge node [left] {} (m-3-3)
            (m-1-3) edge [double equal sign distance,-] node [left] {} (m-2-3)
            (m-2-3) edge node [left] {} (m-3-3)
            ;
        \end{tikzpicture}
    \end{center}
    Now notice that the induced fiber sequence of \cref{remark:fibseqfibers} comes from extending the top fiber sequence to the left.
\end{proof}

\begin{prop}\label{proposition:preservesfibercomposite}
    Given pointed, composable maps $Y \xrightarrow{f} Y' \xrightarrow{g} Y''$,
    if $\degg(S)$-localization preserves the fiber of $g$, and $f$
    is a pointed principal $K(-,n)$-bundle, then $\degg(S)$-localization preserves the fiber of $g \circ f$.
\end{prop}
\begin{proof}
    Consider the square in the beginning of the previous proof and localize it.
    Taking fibers of all the maps in the localized square we claim that we get a diagram equivalent to the following
    \begin{center}
        \begin{tikzpicture}
          \matrix (m) [matrix of math nodes,row sep=2em,column sep=2em,minimum width=2em,nodes={text height=1.75ex,text depth=0.25ex}]
            { \,\,  & L_{\degg(S)} \fib{g}{g(y')} & K(A',n+1)\\
             L_{\degg(S)} Y &  L_{\degg(S)} Y' & K(A',n+1) \\
             L_{\degg(S)} Y'' &  L_{\degg(S)} Y'' & \unit, \\};
          \path[-stealth]
            (m-2-1) edge node [above] {$L_{\degg(S)} f$} (m-2-2)
                    edge node [left] {$L_{\degg(S)}(g\circ f)$} (m-3-1)
            (m-1-2) edge node [right]{} (m-2-2)
                    edge node [above] {} (m-1-3)
            (m-2-2) edge node [left] {$L_{\degg(S)}g$} (m-3-2)
                    edge node [above] {} (m-2-3)
            (m-3-1) edge [double equal sign distance,-] node [left] {} (m-3-2)
            (m-3-2) edge node [left] {} (m-3-3)
            (m-1-3) edge [double equal sign distance,-] node [left] {} (m-2-3)
            (m-2-3) edge node [left] {} (m-3-3)
            ;
        \end{tikzpicture}
    \end{center}
    where $A'$ is the algebraic localization of $A$ away from $S$.
    This is clear for the right vertical fiber sequence and the bottom fiber sequence.
    It is true for the middle vertical fiber sequence by hypothesis, and it is true
    for the middle horizontal fiber sequence by \cref{proposition:preservesprincipalfibrations}.
    Finally, notice that taking the fiber of the top right map we get
    $L_{\degg(S)} \fib{g\circ f}{g(y')}$ by \cref{lemma:inducedfibseqisprincipal} and \cref{proposition:preservesprincipalfibrations},
    so, by commutativity of limits, we see that $\degg(S)$-localization
    preserves the fiber of $g\circ f$.
\end{proof}

\begin{prop}\label{proposition:preservescomposites}
    $\degg(S)$-localization preserves the fiber of pointed maps that factor as composites of pointed principal fibrations.
\end{prop}
\begin{proof}
    We do induction on the number of maps in the factorization.
    The base case can be taken to be the empty factorization, that is, the case when the map is an identity
    map between the same type. This case is immediate.

    For the inductive step, use \cref{proposition:preservesfibercomposite} together with \cref{proposition:preservesprincipalfibrations}.
\end{proof}

\begin{cor}\label{corollary:localizationpreservespostnikov}
    If $X$ is a pointed, nilpotent type, then for every $n: \N$, $\degg(S)$-localization preserves the
    fiber of $\ttrunc{n+1}{X} \to \ttrunc{n}{X}$.
\end{cor}
\begin{proof}
$\degg(S)$-localization preserving the fiber of a map is a mere proposition, so the result follows from \cref{proposition:preservescomposites}.
\end{proof}

\begin{cor}
    The $\degg(S)$-localization of a nilpotent type is nilpotent.
\end{cor}
\begin{proof}
    Given $X$ nilpotent, use \cref{corollary:commutativitylocalizationtruncation} to conclude that
    the $\degg(S)$-localization of the truncation map $\ttrunc{n+1}{X} \to \ttrunc{n}{X}$
    is equivalent to the truncation map $\ttrunc{n+1}{L_{\degg(S)}X} \to \ttrunc{n}{L_{\degg(S)}X}$.
    Now \cref{corollary:localizationpreservespostnikov} implies that this map, being equivalent to
    a $\degg(S)$-localization of a composite of principal fibrations, must be a composite of principal fibrations.
\end{proof}

\begin{thm}
    \label{theorem:localizationlocalizes}
    The $\degg(S)$-localization of a pointed, nilpotent type $X$ localizes all of the homotopy groups away from $S$.
\end{thm}
The proof is now essentially the same as the proof of \citep[Theorem~4.25]{CORS}.

\begin{proof}
    The fact that $\degg(S)$-localization localizes $\pi_1(X)$ is a special case of \citep[Theorem~4.15]{CORS}.
    Now, take $n \geq 1$ and consider the fiber sequence
    \[
        K(\pi_{n+1}(X),n+1) \longhookrightarrow \ttrunc{n+1}{X} \lra \ttrunc{n}{X}.
    \]
    Notice that all the types in the fiber sequence are nilpotent.
    Applying $\degg(S)$-localization we obtain a map of fiber sequences
    \[
        \begin{tikzpicture}
          \matrix (m) [matrix of math nodes,row sep=2em,column sep=3em,minimum width=2em,nodes={text height=1.75ex,text depth=0.25ex}]
            { K(\pi_{n+1}(X),n+1) \strut & \ttrunc{n+1}{X} \strut & \ttrunc{n}{X} \strut \\
            L_{\degg(S)} K(\pi_{n+1}(X),n+1) &  \ttrunc{n+1}{L_{\degg(S)}X} & \ttrunc{n}{L_{\degg(S)} X}\\};
          \path[->]
            (m-1-1) edge [right hook->] node [right] {} (m-1-2)
                    edge node [right] {} (m-2-1)
            (m-1-2) edge node [right] {} (m-1-3)
                    edge node [right] {} (m-2-2)
            (m-2-2) edge node [right] {} (m-2-3)
            (m-2-1) edge [right hook->] node [right] {} (m-2-2)
            (m-1-3) edge node [right] {} (m-2-3)
            ;
        \end{tikzpicture}
    \]
    by \cref{corollary:localizationpreservespostnikov} and \cref{corollary:commutativitylocalizationtruncation}.
    Looking at the bottom fiber sequence, we see that $\pi_n(L_\degg(S)X)$ is the algebraic localization
    of $\pi_n(X)$, by \citep[Theorem~4.22]{CORS}, concluding the proof.
\end{proof}

\section{Fracture squares}
\label{section:fracturesquares}

In this section we construct two fracture squares. The first one holds for any simply connected type.
Its construction is different from the classical construction that can be found in, e.g., \citep{MayPonto}.
A different construction is needed in Homotopy Type Theory if one wants to avoid the use of Whitehead's theorem.
The second square applies to nilpotent types, but requires that they are truncated. Its construction is analogous to the classical construction.

In \cref{fracsimplcon} we show how we can treat simply connected types as generalized abelian groups, and
construct a fracture square for simply connected types with a simple and direct argument.
In \cref{fracnilp} we use Postnikov towers to construct a fracture square for truncated nilpotent types.

\subsection{A fracture theorem for simply connected types}
\label{fracsimplcon}

In this section we will prove one of the many fracture theorems.
The classical version is as follows.
We are given two subrings of the rational numbers $R,S \subseteq \mathbb{Q}$ with $R\cap S = \mathbb{Z}$,
and a pointed simply connected space $X$.
The theorem then says that the following square given by localizations is a homotopy pullback
    \begin{center}
        \begin{tikzpicture}
          \matrix (m) [matrix of math nodes,row sep=2em,column sep=2em,minimum width=2em,nodes={text height=1.75ex,text depth=0.25ex}]
          { X & X_S \\
            X_R & X_T \\};
          \path[-stealth]
            (m-1-1) edge node [above] {} (m-1-2)
                    edge node [left] {} (m-2-1)
            (m-2-1) edge node [above] {} (m-2-2)
            (m-1-2) edge node [right]{} (m-2-2)
            ;
        \end{tikzpicture}
    \end{center}
where $T = R \otimes S$ and the localizations correspond to inverting the units 
in the corresponding ring.

In the constructive setting we let $R,S : \mathbb{N} \to \mathbb{N}$
denote denumerable (multi) sets of natural numbers,
and we define $T : \mathbb{N} \to \mathbb{N}$ by multiplying $R$ and $S$ pointwise.
We also assume that for all $n,m : \N$, $R_n$ is coprime to $S_m$.

We start by constructing a candidate for the fracture square in Homotopy Type Theory.
For this, we need the following.

\begin{lem}
    The maps
    \begin{align*}
        &L_{\degg(R)} X \to L_{\degg(S)}L_{\degg(R)}X,\\
        &L_{\degg(S)}X \to L_{\degg(R)}L_{\degg(S)}X,
    \end{align*}
    and the composites
    \begin{align*}
        &X \to L_{\degg(R)} X \to L_{\degg(S)}L_{\degg(R)}X,\\
        &X\to L_{\degg(S)}X \to L_{\degg(R)}L_{\degg(S)}X, 
    \end{align*}
    are $\degg(T)$-localizations.
    \qed
\end{lem}
\begin{proof}

This follows from \cref{lemma:localizationpreserveslocal} and the fact that for a simply connected type
$\degg(S)$- and $\suspsym \degg(S)$-localization coincide (\cref{proposition:L'simplyconnected}).
\end{proof}

This lets us write $L_{\degg(T)} X$ instead of $L_{\degg(S)}L_{\degg(R)}X$ and $L_{\degg(R)}L_{\degg(S)}X$.
In particular we have a commutative square
\begin{center}
    \begin{tikzpicture}
      \matrix (m) [matrix of math nodes,row sep=2em,column sep=2em,minimum width=2em,nodes={text height=1.75ex,text depth=0.25ex}]
      { X & L_{\degg(S)}X \\
        L_{\degg(R)}X  & L_{\degg(T)}X, \\};
      \path[-stealth]
        (m-1-1) edge node [above] {} (m-1-2)
                edge node [left] {} (m-2-1)
        (m-2-1) edge node [above] {} (m-2-2)
        (m-1-2) edge node [right]{} (m-2-2)
        ;
    \end{tikzpicture}
\end{center}
such that all of its maps are localizations.

\begin{lem}
    \label{lemma:coprimepullback}
    Let $n,m:\mathbb{N}$ be coprime numbers. Then, for every pointed, simply connected type $X$,
    the square
    \begin{center}
        \begin{tikzpicture}
          \matrix (m) [matrix of math nodes,row sep=3em,column sep=4em,minimum width=2em,nodes={text height=1.75ex,text depth=0.25ex}]
          { \loopspacesym X & \loopspacesym X \\
            \loopspacesym X & \loopspacesym X \\};
          \path[-stealth]
            (m-1-1) edge node [above] {$n$} (m-1-2)
                    edge node [left] {$m$} (m-2-1)
            (m-2-1) edge node [above] {$n$} (m-2-2)
            (m-1-2) edge node [right]{$m$} (m-2-2)
            ;
        \end{tikzpicture}
    \end{center}
    is a pullback.
\end{lem}
\begin{proof}
    Let $\alpha n + \beta m = 1$. Since $X$ is connected, $\loopspacesym$ reflects equivalences,
    and since it also preserves limits, the square in the statement is a pullback if and only if it is after looping.
    Let $Y \defeq \loopspacesym^2 X$. It is then enough to show that
    \[
        Y \xrightarrow{\begin{pmatrix}m \\ n\end{pmatrix}} Y \times Y \xrightarrow{\begin{pmatrix}n & -m\end{pmatrix}} Y
    \]
    is a fiber sequence, where we are considering the fiber over $\refl{}$.
    Here we are using matrix notation, so that, for example, the map
    ${\begin{pmatrix}n & -m\end{pmatrix}} : Y \times Y \to Y$ is given by $(l,l') \mapsto (l^n \sq l'^{-m})$.
    The reason why it is enough to consider the above fiber sequence is that the fiber of $\begin{pmatrix}n & -m\end{pmatrix}$
    is $\sum_{(l,l' : Y)} l^n \sq l'^{-m} = \refl{}$, which is equivalent to the pullback of
    $Y \xrightarrow{n} Y \xleftarrow{m} Y$, namely $\sum_{(l,l' : Y)} l^n = l'^{m}$.

    Notice that we have maps
    \[
     \begin{pmatrix}n & -m \\\beta & \alpha \end{pmatrix},
     \begin{pmatrix}\alpha & m \\-\beta & n \end{pmatrix} :
         Y \times Y \to Y\times Y
    \]
    and these are mutual inverses, using the fact that $Y$ is homotopy commutative.

    This means that we have an equivalence of maps
    \begin{center}
        \begin{tikzpicture}
          \matrix (m) [matrix of math nodes,row sep=3em,column sep=4em,minimum width=2em,nodes={text height=1.75ex,text depth=0.25ex}]
          { Y \times Y & Y \\
            Y \times Y & Y \\};
          \path[-stealth]
            (m-1-1) edge node [above] {$\begin{pmatrix}n & -m\end{pmatrix}$} (m-1-2)
                    edge node [left] {$\begin{pmatrix}n & -m \\\beta & \alpha \end{pmatrix}$} (m-2-1)
            (m-2-1) edge node [above] {$\begin{pmatrix}1 & 0\end{pmatrix}$} (m-2-2)
            (m-1-2) edge [double equal sign distance,-] node [right] {} (m-2-2)
            ;
        \end{tikzpicture}
    \end{center}
    But the fiber of the bottom map is clearly $Y$,
    so using the inverse of the left vertical map we get the fiber sequence that we needed.
\end{proof}

\begin{rmk}
Define $r,s,t,\rho,\sigma: \mathbb{N} \to \mathbb{N}$ by:
\[
   r_n := \prod_{i = 0}^n R_i,\,\,\,\,\,\,\, s_n := \prod_{i = 0}^n S_i,\,\,\,\,\,\,\, t_n := r_n \times s_n,
   \,\,\,\,\,\,\, \rho_n := \prod_{i = 0}^n r_i, \,\,\,\,\,\,\, \sigma_n := \prod_{i = 0}^n s_i.
\]
Then if $H$ is a loop space, the following diagram commutes:
\begin{equation}
    \label{diagram:bigdiagram}
    \begin{tikzpicture}[baseline=(current  bounding  box.center)]
          \matrix (m) [matrix of math nodes,row sep=3em,column sep=3em,minimum width=2em,nodes={text height=1.75ex,text depth=0.25ex}]
          { H & H & H & H & \cdots \\
            H & H & H & H & \cdots \\
            H & H & H & H & \cdots \\};
          \path[-stealth]
            (m-1-1) edge node [above] {$r_1$} (m-1-2)
                    edge node [left] {$\sigma_0$} (m-2-1)
            (m-2-1) edge node [above] {$t_1$} (m-2-2)
            (m-3-1) edge node [above] {$s_1$} (m-3-2)
                    edge node [left] {$\rho_0$} (m-2-1)
            (m-1-2) edge node [above] {$r_2$} (m-1-3)
                    edge node [left] {$\sigma_1$} (m-2-2)
            (m-2-2) edge node [above] {$t_2$} (m-2-3)
            (m-3-2) edge node [above] {$s_2$} (m-3-3)
                    edge node [left] {$\rho_1$} (m-2-2)
            (m-1-3) edge node [above] {$r_3$} (m-1-4)
                    edge node [left] {$\sigma_2$} (m-2-3)
            (m-2-3) edge node [above] {$t_3$} (m-2-4)
            (m-3-3) edge node [above] {$s_3$} (m-3-4)
                    edge node [left] {$\rho_2$} (m-2-3)
            (m-1-4) edge node [above] {$r_4$} (m-1-5)
                    edge node [left] {$\sigma_3$} (m-2-4)
            (m-2-4) edge node [above] {$t_4$} (m-2-5)
            (m-3-4) edge node [above] {$s_4$} (m-3-5)
                    edge node [left] {$\rho_3$} (m-2-4)
            ;
        \end{tikzpicture} \end{equation}
\end{rmk}

\begin{thm}
    \label{theorem:fracturetheorem}
    If $X$ is simply connected and for all $n,m : \N$, $R_n$ is coprime to $S_m$, then the square
\begin{center}
    \begin{tikzpicture}
      \matrix (m) [matrix of math nodes,row sep=2em,column sep=2em,minimum width=2em,nodes={text height=1.75ex,text depth=0.25ex}]
      { X & L_{\degg(S)}X \\
        L_{\degg(R)}X  & L_{\degg(T)}X, \\};
      \path[-stealth]
        (m-1-1) edge node [above] {} (m-1-2)
                edge node [left] {} (m-2-1)
        (m-2-1) edge node [above] {} (m-2-2)
        (m-1-2) edge node [right]{} (m-2-2)
        ;
    \end{tikzpicture}
\end{center}
    is a pullback.
\end{thm}
\begin{proof}
    Since we must prove a mere proposition and $X$ is connected, we can assume that we have $x : X$.
    Since $X$ is connected, and the loop functor from pointed connected types to types preserves limits
    and reflects equivalences, the square is a pullback if and only if
    it is after looping.
    Using the fact that $X$ is simply connected, we see that applying loop space to the diagram, we get a diagram
    equivalent to
    \begin{center}
        \begin{tikzpicture}
          \matrix (m) [matrix of math nodes,row sep=2em,column sep=2em,minimum width=2em,nodes={text height=1.75ex,text depth=0.25ex}]
          { \loopspacesym X & L_{\degg(S)}\loopspacesym X \\
              L_{\degg(R)}\loopspacesym X &  L_{\degg(T)}\loopspacesym X \\};
          \path[-stealth]
            (m-1-1) edge node [above] {} (m-1-2)
                    edge node [left] {} (m-2-1)
            (m-2-1) edge node [above] {} (m-2-2)
            (m-1-2) edge node [right]{} (m-2-2)
            ;
        \end{tikzpicture}
    \end{center}
    by \citep[Corollary~4.13]{CORS}.

    By taking $H$ to be $\loopspacesym X$ in \cref{diagram:bigdiagram} and taking the 
    colimit of each of the rows we get the diagram:
    \[
        L_{\degg(R)}\loopspacesym X \rightarrow L_{\degg(T)}\loopspacesym X \leftarrow L_{\degg(S)}\loopspacesym X,
    \]
    by \citep[Theorem~4.20]{CORS}.

    Observe that by hypothesis $\rho_n$ and $\sigma_n$ are coprime, so by \cref{lemma:coprimepullback},
    if we instead take the pullback in each column we get a diagram equivalent to:
    \[
    \loopspacesym X \to \loopspacesym X \to \loopspacesym X \to \loopspacesym X \to \cdots
    \]
    Let us assume for a moment that the maps in this diagram are the identity.
    Then the colimit of this diagram is $\loopspacesym X$,
    so the result follows from the commutativity of pullbacks and filtered colimits \citep{DoornRijkeSojakova}.

    To prove that the maps are the identity observe that, by \cref{lemma:coprimepullback},
    in the $n$-th place we have a map of pullbacks:
    \begin{center}
        \begin{tikzpicture}
          \matrix (m) [matrix of math nodes,row sep=3em,column sep=2em,minimum width=2em,nodes={text height=1.75ex,text depth=0.25ex}]
          { \loopspacesym X & & & \loopspacesym X \\
               &  \loopspacesym X & \loopspacesym X & \\
               &  \loopspacesym X & \loopspacesym X & \\
            \loopspacesym X & & & \loopspacesym X \\};
          \path[-stealth]
            (m-1-1) edge [dashed] node [above] {} (m-1-4)
                    edge node [left] {$\sigma_n$} (m-4-1)
                    edge node [above] {$\,\,\,\,\,\,\,\rho_n$} (m-2-2)
            (m-4-1) edge node [below] {$\,\,\,\,\,\,\,\rho_n$} (m-3-2)
                    edge node [below] {$s_{n+1}$} (m-4-4)
            (m-2-2) edge node [left] {$\sigma_n$} (m-3-2)
                    edge node [above] {$r_{n+1}$} (m-2-3)
            (m-3-2) edge node [above] {$t_{n+1}$} (m-3-3)
            (m-1-4) edge node [above] {$\rho_{n+1}\,\,\,\,\,\,$} (m-2-3)
                    edge node [right] {$\sigma_{n+1}$} (m-4-4)
            (m-2-3) edge node [right] {$\sigma_{n+1}$} (m-3-3)
            (m-4-4) edge node [below] {$\rho_{n+1}\,\,\,\,\,\,$} (m-3-3)
            ;
        \end{tikzpicture}
    \end{center}
    where the pullbacks are the left and right squares, and the dotted map is 
    the map we want to show is the identity. Since this map is induced by the universal
    property of pullbacks, it is enough to check that the identity makes the cube commute,
    which is straightforward.
\end{proof}

\subsection{A fracture theorem for truncated nilpotent types}
\label{fracnilp}

We now prove an analogous fracture theorem, but for truncated nilpotent types.
Getting rid of the truncation hypothesis remains an open problem.

We start with a fracture theorem for abelian groups.
\begin{prop}
    Given an abelian group $A$ and $n \geq 1$, if for all $n,m : \N$, $R_n$ is coprime to $S_m$, then the square
    \begin{center}
        \begin{tikzpicture}
          \matrix (m) [matrix of math nodes,row sep=2em,column sep=2em,minimum width=2em,nodes={text height=1.75ex,text depth=0.25ex}]
            { K(A,n) & L_{\degg(S)}K(A,n) \\
              L_{\degg(R)}K(A,n) & L_{\degg(T)}K(A,n) \\};
          \path[-stealth]
            (m-1-1) edge node [above] {} (m-1-2)
                    edge node [left] {} (m-2-1)
            (m-2-1) edge node [above] {} (m-2-2)
            (m-1-2) edge node [right]{} (m-2-2)
            ;
        \end{tikzpicture}
    \end{center}
    is a pullback.
\end{prop}
\begin{proof}
    Use \cref{theorem:fracturetheorem}, taking $X$ to be $K(A,n+1)$.
    Looping the pullback square, we get a pullback square.
    To conclude use the fact that, since $A$ is abelian, we have
    $\loopspacesym\left(L_{\degg(R)}K(A,n+1)\right)\simeq L_{\degg(R)}K(A,n)$ by \citep[Lemma~4.13]{CORS},
    and likewise for $S$ and $T$.
\end{proof}

\begin{thm}
    If $X$ is truncated and nilpotent and for all $n,m : \N$, $R_n$ is coprime to $S_m$, then the square
    \begin{center}
        \begin{tikzpicture}
          \matrix (m) [matrix of math nodes,row sep=2em,column sep=2em,minimum width=2em,nodes={text height=1.75ex,text depth=0.25ex}]
          { X & L_{\degg(S)}X \\
              L_{\degg(R)}X & L_{\degg(T)}X \\};
          \path[-stealth]
            (m-1-1) edge node [above] {} (m-1-2)
                    edge node [left] {} (m-2-1)
            (m-2-1) edge node [above] {} (m-2-2)
            (m-1-2) edge node [right]{} (m-2-2)
            ;
        \end{tikzpicture}
    \end{center}
    is a pullback.
\end{thm}

The following proof is classical, and appears in \citep{MayPonto}.

\begin{proof}
    Since we have to prove a mere proposition, we can assume given a nilpotent structure for $X$.
    We proceed by induction on the nilpotency degree of $X$.
    The base case is clear.
    For the inductive step, let $K(A,n) \to X \to X'$ be a principal fibration,
    and consider its classifying map $X' \to K(A,n+1)$.
    We can form the following diagram 
    \begin{center}
        \begin{tikzpicture}
          \matrix (m) [matrix of math nodes,row sep=2em,column sep=2em,minimum width=2em,nodes={text height=1.75ex,text depth=0.25ex}]
            { L_{\degg(R)}X' & L_{\degg(R)}K(A,n+1) & \unit \\
              L_{\degg(T)}X' & L_{\degg(T)}K(A,n+1) & \unit \\
              L_{\degg(S)}X' & L_{\degg(S)}K(A,n+1) & \unit. \\};
          \path[-stealth]
            (m-1-1) edge node [above] {} (m-1-2)
                    edge node [left] {} (m-2-1)
            (m-2-1) edge node [above] {} (m-2-2)
            (m-1-2) edge node [right]{} (m-2-2)
            (m-3-1) edge node [right] {} (m-2-1)
                    edge node [right] {} (m-3-2)
            (m-3-2) edge node [right] {} (m-2-2)
            (m-1-3) edge node [above] {} (m-1-2)
                    edge node [right] {} (m-2-3)
            (m-3-3) edge node [right] {} (m-2-3)
                    edge node [above] {} (m-3-2)
            (m-2-3) edge node [right] {} (m-2-2)

            ;
        \end{tikzpicture}
    \end{center}
    Using \cref{proposition:preservesprincipalfibrations}, we see that
    taking the pullback of each row gives us the cospan $L_{\degg(R)}X \rightarrow L_{\degg(T)}X \leftarrow L_{\degg(S)}X$.
    On the other hand, by taking the pullback of each column, we get the cospan $X' \rightarrow K(A,n) \leftarrow \unit$.
    So by the commutativity of limits, we are done.
\end{proof}

\section{Conclusions}

We have developed the theory of nilpotent types in Homotopy Type Theory,
including some of their cohomological properties
and the basic properties of their localizations away from sets of numbers, such as
the effect of localization on their homotopy groups.

As part of this, we developed the theory of $K(A,n)$-bundles, and their induced group actions.
We saw how Homotopy Type Theory helps in seeing this theory as a generalization of the theory of
groups and group actions.

\bibliographystyle{dcu}
\bibliography{references}

\end{document}